\documentclass[3p,times,twocolumn]{elsarticle}   

\usepackage[utf8]{inputenc}
\usepackage{graphicx}          
\graphicspath{ {Figures_bnd_wave/} }
\usepackage{amsmath}
\usepackage{natbib}
\usepackage{amsfonts}
\usepackage{color}
\usepackage{amssymb}
\usepackage{mathrsfs}
\usepackage{float}
\usepackage{enumerate}
\usepackage{verbatim}
\usepackage{setspace}
\usepackage{epsfig}
\usepackage{url}
\usepackage{graphicx}
\usepackage{lscape}
\usepackage{subcaption}
\usepackage{caption}
\usepackage{bbm}
\usepackage{arydshln}

\numberwithin{equation}{section}
\newtheorem{remark}{Remark}[section]

\newtheorem{theorem}{Theorem}[section]

\newtheorem{lemma}{Lemma}[section]

\newenvironment{proof}{{\textbf{Proof}:~}}{\hfill$\Box$\\}

\begin{document}

\begin{frontmatter}
\title{Delayed finite-dimensional observer-based control of 1D parabolic PDEs via reduced-order LMIs\tnoteref{t1}}

\tnotetext[t1]{Supported by Israel Science Foundation (grant  673/19), the C. and H. Manderman Chair 
at Tel Aviv University and by the Y. and C. Weinstein Research
Institute for Signal Processing }
\author[Tel-Aviv university]{Rami Katz}\ead{rami@benis.co.il}
\author[Tel-Aviv university]{Emilia Fridman}\ead{emilia@eng.tau.ac.il}
\address[Tel-Aviv university]{School of Electrical Engineering, Tel-Aviv University, Tel-Aviv}  
%
%
%

\begin{abstract}
Recently a constructive method was introduced for  finite-dimensional observer-based control of  1D parabolic PDEs.
In this paper we present an improved method in terms of the reduced-order LMIs (that significantly shorten the computation time) and introduce predictors to manage with larger delays.
We treat the case of a 1D heat equation under Neumann actuation and non-local measurement, that has not been studied yet.
We apply modal decomposition  and prove $L^2$ exponential stability  by a direct Lyapunov method.
We provide reduced-order LMI conditions for finding the observer dimension $N$ and resulting decay rate.
The LMI dimension 
does not grow with $N$. The LMI  is always
feasible for large $N$, and feasibility for $N$ implies feasibility for $N+1$.
For the first time we manage with delayed implementation of the controller in the presence of  fast-varying  (without any constraints
on the delay-derivative) input and  output delays. To manage with larger delays, we
construct classical observer-based predictors. For the known input delay, the LMIs dimension does not grow with $N$, whereas
for unknown one the LMIs dimension grows, but it is  essentially smaller than in the existing results.
A numerical example demonstrates the efficiency of our method.
\end{abstract}
\end{frontmatter}

\section{Introduction}
{Observer-based controllers for  PDEs with  observers in the form of PDEs have been constructed  in  \cite{curtain1982finite,lasiecka2000control,
		Krstic2008} (to name a few). 
	Very attractive for practical applications finite-dimensional observer-based  controllers for parabolic systems was studied by using the modal decomposition approach  in \cite{balas1988finite,christofides2001,curtain1982finite,ghantasala2012active,harkort2011finite}.
	%
	{\color{blue} The recent papers \cite{RamiContructiveFiniteDim,Rami_CDC20,katz2020finite} on constructive LMI-based finite-dimensional  observer-based control have introduced $N$-dimensional observers, where the gains (as well as the controller gains) are based only  on the $N_0\le N$ unstable modes.
		However, the stability analysis was based on the full-order closed-loop systems. The latter led to higher-order LMIs whose dimension grows with $N$ and complicated proofs of their feasibility.}
	
	Delayed and/or sampled-data finite-dimensional controllers  were designed in \cite{
		Aut12,kang2018distributed,
		selivanov2019delayed} for distributed static output-feedback control and in \cite{
		karafyllis2018sampled,espitia2020event} for boundary state-feedback 
	control. 
	Delayed implementation of finite-dimensional observer-based controllers for the 1D heat equation
	was presented in \cite{katz2020constructiveDelay}.
	In the case of Dirichlet actuation considered in \cite{katz2020constructiveDelay}, the results were not applicable to the case where both input and output delays are fast-varying (without any constraints on the time-derivative that correspond e.g. to sampled-data and network-based control).
	{\color{blue} For boundary control in the presence of fast-varying input and output delays only infinite-dimensional PDE observers have been suggested till now \cite{katz2020boundary}.}
	%

	Large 
	input delays for PDEs can be compensated by classical predictors
	\cite{Krstic09}.
	Predictor-based controllers for ODEs that compensated an arbitrary large constant part of a delay were suggested in  \cite{karafyllis2017predictor,Mazenc13,selivanov2016observer} and
	extended to state-feedback boundary control of parabolic PDEs in \cite{lhachemi2019lmi,prieur2018feedback}.
	For coupled systems of ODEs, predictors may enlarge the constant part of the delay which preserves stability, but cannot manage
	with arbitrary large constant delays due to coupling \cite{liu2018distributed,zhu2020observer}. {\color{blue} However, the finite-dimensional observer-based predictors have not been constructed yet for PDEs.} 

In the present paper, we introduce finite-dimensional observer-based controllers for the 1D heat equation under Neumann actuation and non-local measurement.
We apply modal decomposition  to the original system (without dynamic extension)
and prove $L^2$ exponential
stability of the closed-loop system by a direct Lyapunov method.
{\color{blue} The paper  {
	contribution} to challenging finite-dimensional observer-based control can be summarized as follows:
\begin{enumerate}
	\item
	The paper introduces {\it reduced-order closed-loop} system that reveals the {\it singularly perturbed structure} of the system,  leads to 
	{\it reduced-order LMIs}, trivializes the LMIs feasibility proof and the fact that their feasibility  for the observer dimension $N$  implies  feasibility for $N+1$.
	In example,  the feasibility of the reduced-order LMIs for the delayed case
	can be easily verified for 
	$N=30$, whereas  in \cite{katz2020constructiveDelay} the corresponding conditions  could not be verified for $N=9$. Note that larger $N$ enlargers  delays that preserve the stability.
	\item 
	For the first time in the
	case of boundary  control, the results
	are applicable to {\it fast-varying input and output delays}.
	This is because the proportional controller under Neumann actuation and non-local measurement  leads to $L^2$ convergence. 
	For briefness, our results are presented for differentiable delays. However,
	via the time-delay approach to networked control \cite{Fridman14_TDS}, the same LMI conditions are applicable to networked control implementation via a zero-order-hold device, under sampled-data delayed measurements. 
	\item
	The first {\it finite-dimensional observer-based predictor} is constructed to compensate the constant part of input fast-varying delay, and this is  in the presence of the small output fast-varying  delay. We
	present the classical predictors using the reduction approach \cite{Artstein82}. We predict the future state of the  observer, whereas
	the infinite-dimensional part depends on the uncompensated large delay. 
	We consider the case of either known or unknown input delay.  For the known input delay, the LMIs dimension does not grow with $N$, whereas
	for the unknown one it grows, but is essentially smaller than in \cite{katz2020constructiveDelay}.
	An example demonstrates the efficiency of the method and shows that predictors allow for larger delays which preserve the stability.
\end{enumerate}}
Our new method 
can be applied to other classes of parabolic PDEs (see Remark \ref{rem_other PDEs} below). In the conference version of the paper \cite{katz2020Reduced} predictors were not considered.

\emph{Notations and preliminaries:}
$ L^2(0,1)$ is the Hilbert space of Lebesgue measurable and square integrable functions $f:[0,1]\to \mathbb{R} $  with inner product $\left< f,g\right>:=\scriptsize{\int_0^1 f(x)g(x)dx}$ and norm $\left\|f \right\|^2:=\left<f,f \right>$.
$H^k(0,1)$ is the Sobolev space of  functions  $f:[0,1]\to \mathbb{R} $ having $k$ square integrable weak derivative, with norm $\left\|f \right\|^2_{H^k}:=\textcolor{blue}{\sum_{j=0}^{k} \left\|\frac{\text{d}^jf}{\text{dx}^j} \right\|^2}$.
The Euclidean norm on $\mathbb{R}^n$ will be denoted by $\left|\cdot \right|$.
For  $P \in \mathbb{R}^{n \times n}$, the notation $P>0$ means that $P$ is symmetric and positive definite.
The sub-diagonal elements of a symmetric matrix are denoted by $*.$ For $U\in \mathbb{R}^{n\times n}, \ U>0$ and $X\in \mathbb{R}^n$ we denote $\left|X\right|^2_U=X^TUX$. We denote by $\mathbb{Z}_+$ the set of nonnegative integers.

Recall that the Sturm-Liouville eigenvalue problem
\begin{equation}\label{eq:SL}
\begin{array}{lll}
	&\phi''+\lambda \phi = 0,\ \ x\in [0,1]\quad ; \quad \phi'(0)=\phi'(1)=0,
\end{array}
\end{equation}
induces a sequence of eigenvalues $\scriptsize \lambda_n = n^2\pi^2, n\geq 0$ with corresponding eigenfunctions
\begin{equation}\label{eq:Eigenfunctions}
\begin{array}{lll}
	\phi_0(x) = 1, \quad
	\phi_n(x)=\sqrt{2}\cos\left(\sqrt{\lambda_n} x\right), n\geq 1.
\end{array}
\end{equation}
Moreover, the eigenfunctions form a complete orthonormal system in $L^2(0,1)$.
Given $N\in \mathbb{Z}_+$ and $h\in L^2(0,1)$ satisfying $h \overset{L^2}{=} \sum_{n=0}^{\infty}h_n\phi_n$ we will use
the notation $\left\|h\right\|_N^2 =\left\|h\right\|^2-\sum_{n=0}^Nh_n^2=\sum_{n=N+1}^{\infty}h_n^2$.

\section{Non-delayed $L^2$-stabilization}\label{Sec:NoDel}
Consider the reaction-diffusion system
\begin{equation}\label{eq:PDE1PointActNonDelayed}
	\begin{aligned}
		& z_t(x,t)=z_{xx}(x,t)+qz(x,t),\ z_x(0,t)=0, \ z_x(1,t)=u(t)
	\end{aligned}
\end{equation}
where $t\geq 0$, $x\in [0,1]$, $z(x,t)\in \mathbb{R}$ and $q\in \mathbb{R}$ is the reaction coefficient. We consider Neumann actuation with a control input $u(t)$ and non-local measurement of the form
\begin{equation}\label{eq:BoundMeasNonDelayed}
	y(t) = \left<c,z(\cdot,t)\right>, \quad c\in L^2(0,1).
\end{equation}

{Below, we prove the existence and uniqueness of a classical solution to \eqref{eq:PDE1PointActNonDelayed} (see proof after \eqref{eq:ContDef}). Therefore, we can present the solution as}
\begin{equation}\label{eq:zSolDecomp}
	z(x,t)\overset{L^2}{=}\sum_{n=0}^{\infty}z_n(t)\phi_n(x), \ \  z_n(t) = \left<z(\cdot,t),\phi_n\right>.
\end{equation}
with $\phi_n(t), \ n\in \mathbb{Z}_+$ given in \eqref{eq:Eigenfunctions} {(see e.g \cite{christofides2001,karafyllis2018sampled})}. {\color{blue}
	Differentiating $z_n(t)$ and substituting $z_t=z_{xx}+qz$  we have
	\begin{equation*}
		\begin{array}{lll}
			\dot{z}_n(t)\!=\!\int_0^1z_t(x,t)\phi_n(x)dx\!=
			\int_0^1z_{xx}(x,t)\phi_n(x)dx +qz_n(t).
		\end{array}
	\end{equation*}
	Integrating by parts twice and using the boundary conditions for $z$ and $\phi_n$ we find
	\begin{equation*}
		\begin{array}{lll}
			&\int_0^1z_{xx}(x,t)\phi_n(x)dx = -\lambda_nz_n(t)+\phi_n(1)u(t)
		\end{array}
	\end{equation*}
	which leads to}
\begin{equation}\label{eq:ZOdesPointActNonDelayed}
	\begin{aligned}
		&\dot{z}_n(t) = (- \lambda_n + q) z_n(t) + b_n u(t), \quad t\geq 0, \\
		&b_0 = 1, \ \ b_n = {(-1)}^n \sqrt{2}, \ n\in \mathbb{Z}_+.
	\end{aligned}
\end{equation}
In particular, note that
\begin{equation}\label{eq:bnNeq0}
	b_n \neq 0, \quad n\in \mathbb{Z}_+
\end{equation}
and for $N\geq 0$ the following holds:
\begin{equation}\label{eq:IntegTest}
	\begin{array}{lll}
		&\hspace{-3mm}\sum_{n=N+1}^{\infty}b_n^2\lambda_n^{-1} =  \frac{2}{\pi^2} \sum_{n=N+1}^{\infty} \frac{1}{n^2}\leq  \frac{2}{\pi^2N}.
	\end{array}
\end{equation}	
Let $\delta>0$ be a desired decay rate. Since $\lim_{n \to \infty}\lambda_n=\infty$, there exists some $N_0 \in \mathbb{Z}_+$ such that
\begin{equation}\label{eq:N0}
	-\lambda_n+q<-\delta, \quad n>N_0.
\end{equation}	
Let $N\geq N_0+1$, where  $N$ will define the dimension of the observer, whereas $N_0$ will be the dimension of the controller. We construct a $N$-dimensional observer of the form
\begin{equation}\label{eq:ZhatSeries0}
	\hat{z}(x,t): = \sum_{n=0}^{N}\hat{z}_n(t)\phi_n(x)
\end{equation}
where $\hat{z}_n(t)$ satisfy the ODEs for $t\geq 0$
\begin{equation}\label{eq:obsODENonDelayed}
	\begin{array}{lll}
		\dot{\hat{z}}_n(t) &= (-\lambda_n+q)\hat{z}_n(t) + b_nu(t)\\
		&-l_n\left[\left<\sum_{n=0}^N\hat{z}_n(t)\phi_n ,c\right> - y(t)\right],\\
		\hat{z}_n(0)&=0, \quad 0\leq n\leq N.
	\end{array}
\end{equation}
Here $l_n, \ 0\leq n\leq N$ are scalars, and $l_{N_0+1}=...=l_N=0$.
We further choose $l_{N_0+1}=...=l_N=0$. \textcolor{blue}{This choice will lead to a reduced-order closed-loop system (see \eqref{eNN0}, \eqref{eq:ClosedLoopReduced} below) with omitted ODEs for $\hat z_{N_0+1},..., \hat z_N$ and 
	will not deteriorate the performance of the closed-loop system.}
%
Let
\begin{equation}\label{eq:C0A0}
\begin{array}{ll}
&\hspace{-4mm}A_0 = \operatorname{diag}\left\{-\lambda_i+q \right\}_{i=0}^{N_0},\ L_0 = \operatorname{col}\left\{l_i\right\}_{i=0}^{N_0}\\
&\hspace{-4mm}B_0 = \operatorname{col}\left\{b_i\right\}_{i=0}^{N_0},\ C_0=\left[c_0,\dots,c_{N_0} \right], \ c_n=\left<c,\phi_n\right>.
\end{array}
\end{equation}
Assume that
\begin{equation}\label{eq:AsscnNonDelayed}
c_n \neq 0 , \quad 0\leq n \leq N_0.
\end{equation}
By the Hautus lemma $(A_0,C_0)$ is observable. We choose $L_0 = [l_0,\dots, l_{N_0}]^T$ which satisfies the  Lyapunov inequality:
\begin{equation}\label{eq:GainsDesignL}
P_{\text{o}}(A_0-L_0C_0)+(A_0-L_0C_0)^TP_{\text{o}} < -2\delta P_{\text{o}}
\end{equation}
with $0<P_{\text{o}}\in \mathbb{R}^{(N_0+1)\times (N_0+1)}$. By the Hautus lemma,  $(A_0,B_0)$ is controllable due to \eqref{eq:bnNeq0}. Let $K_0\in \mathbb{R}^{1\times (N_0+1)}$ satisfy the Lyapunov inequality
\begin{equation}\label{eq:GainsDesignK}
\begin{array}{lll}
&P_{\text{c}}(A_0+B_0K_0)+(A_0+B_0K_0)^TP_{\text{c}} < -2\delta P_{\text{c}},
\end{array}
\end{equation}
where $0<P_{\text{c}}\in \mathbb{R}^{(N_0+1)\times (N_0+1)}$. We propose a 
controller
\begin{equation}\label{eq:ContDef}
\begin{aligned}
u(t)= K_0\hat{z}^{N_0}(t),\quad
\hat{z}^{N_0}(t) = \left[\hat{z}_0(t),\dots,\hat{z}_{N_0}(t) \right]^T
\end{aligned}
\end{equation}
which is based on the $N$-dimensional observer \eqref{eq:obsODENonDelayed}. Note that \eqref{eq:obsODENonDelayed} implies $u(0)=0$.

{For well-posedness we introduce the change of variables $w(x,t)=z(x,t)-
	\frac{1}{2}x^2u(t)$ leading to the equivalent PDE
	\begin{equation}\label{eq:PDEChangeVar}
		\begin{array}{lll}
			&\hspace{-5mm}w_t(x,t)=w_{xx}(x,t)+qw(x,t)+f(x,t), \ x\in [0,1], \ t\geq 0,\\
			&\hspace{-5mm}f(x,t) = -\frac{1}{2}x^2\dot{u}(t)+\left(\frac{q}{2}x^2+1\right)u(t),\\
			&\hspace{-5mm}w_x(0,t)=0, \quad w_x(1,t)=0.
		\end{array}
	\end{equation}
	Consider the operator
	\begin{equation}\label{eq:CalADef}
		\begin{array}{lll}
			&\textcolor{blue}{\mathfrak{A}}:\mathcal{D}(\textcolor{blue}{\mathfrak{A}})\subseteq L^2(0,1)\to L^2(0,1), \ \ \textcolor{blue}{\mathfrak{A}}h = -h'',\\
			&\mathcal{D}(\textcolor{blue}{\mathfrak{A}}) = \left\{h\in H^2(0,1)| h'(0)=h'(1)=0\right\}.
		\end{array}
	\end{equation}
	It is well known that $\textcolor{blue}{\mathfrak{A}}$ generates a strongly continuous semigroup on $L^2(0,1)$ \cite{pazy1983semigroups}. 
	Let $\textcolor{blue}{\mathbb{G}}=L^2(0,1)\times \mathbb{R}^{N+1}$ be a Hilbert space with the norm $\left\|\cdot\right\|_{\textcolor{blue}{\mathbb{G}}}=\sqrt{\left\|\cdot\right\|+\left|\cdot\right|}$.
	Defining the state $\xi(t)=\text{col}\left\{w(\cdot,t),\hat{z}^N(t)\right\}$, where
	\begin{equation}\label{hatzN}
		\begin{array}{lll}
			\hat{z}^{N}(t)=\text{col}\left\{\hat{z}_0(t),\dots,\hat{z}_N(t)\right\}
		\end{array}
	\end{equation}
	the closed-loop system \eqref{eq:obsODENonDelayed}, \eqref{eq:ContDef} and \eqref{eq:PDEChangeVar} can be presented as
	\begin{equation*}
		\frac{\text{d}}{\text{dt}}\xi(t)+\operatorname{diag}\left\{\textcolor{blue}{\mathfrak{A}},\textcolor{blue}{\mathfrak{B}}\right\}\xi(t) = \operatorname{col}\left\{f_1(\xi),f_2(\xi)\right\}
	\end{equation*}
	where
	\begin{equation*}
		\begin{array}{lll}
			&\textcolor{blue}{\mathfrak{B}}\xi_2 \scriptsize= \begin{bmatrix}
				-\left(A_0+B_0K_0-L_0C_0 \right) & L_0C_1\\-B_1K_0 & -A_1
			\end{bmatrix}\xi_2,\ \ \xi_2\in \mathbb{R}^{N+1},\\
			& f_1(\xi) = \scriptsize \begin{bmatrix}
				q & \upsilon & \frac{x^2}{2}K_0L_0C_1
			\end{bmatrix}\normalsize \xi-\frac{x^2}{2}K_0L_0\left<c,\xi_1 \right>,\\
			&f_2(\xi) = \text{col}\left\{L_0\left<c,\xi_1 \right>+\frac{1}{2}\left<c,x^2 \right>K_0\xi_2, 0\right\}, \\
			&\upsilon = \left(\frac{q}{2}x^2+1\right)K_0-\frac{x^2}{2}K_0(A_0+B_0K_0-L_0C_0)\\
			&\hspace{5mm}+\frac{1}{2}\left<c,x^2\right>L_0K_0.
		\end{array}
	\end{equation*}
	$f_1$ and $f_2$ are linear and, therefore, continuously differentiable. Let $z(\cdot,0)=w(\cdot,0)\in H^1(0,1)$. By Theorems  6.3.1 and 6.3.3 in \cite{pazy1983semigroups}, there exists a unique classical solution
	\begin{equation}\label{eq:zetaWP}
		\xi \in C\left([0,\infty);\textcolor{blue}{\mathbb{G}}\right)\cap C^1\left((0,\infty);\textcolor{blue}{\mathbb{G}}\right)
	\end{equation}
	satisfying $\xi(t) \in \mathcal{D}\left(\textcolor{blue}{\mathfrak{A}}\right)\times \mathbb{R}^{N+1},\ t>0$. Applying $z(x,t)=w(x,t)+\frac{1}{2}x^2u(t)$, \eqref{eq:PDE1PointActNonDelayed} and \eqref{eq:obsODENonDelayed}, subject to \eqref{eq:ContDef}, have a unique classical solution such that $z\in C([0,\infty),L^2(0,1))\cap C^1((0,\infty),L^2(0,1))$ and $z(\cdot,t)\in H^2(0,1)$ with $z_x(0,t)=0, \ z_x(1,t)=u(t)$ for $t\in[0,\infty)$.}

Let
\begin{equation}\label{eq:EstErrorNonDelayed0}
	e_n(t) = z_n(t)-\hat{z}_n(t), \ 0\leq n \leq N
\end{equation}
be the estimation error. The last term on the right-hand side of \eqref{eq:obsODENonDelayed} can be written as
\begin{equation}\label{eq:IntroZetaNonDelayed}
	\begin{array}{ll}
		&\int_0^1 c(x)\left[\sum_{n=1}^N\hat{z}_n(t)\phi_n(x)-\sum_{n=1}^{\infty}z_n(t)\phi_n(x)\right]dx \\
		&=-\sum_{n=0}^{N} c_ne_n(t)-\zeta(t), \ \zeta(t)=\sum_{n=N+1}^{\infty}c_n z_n(t).
	\end{array}
\end{equation}
Then the error equations for $0\leq n \leq N$ and $ t\ge 0$ are 
\begin{equation}\label{eq:en0}
	\begin{array}{r}
		\dot e_n(t)=(-\lambda_n+q)e_n(t)
		-l_n\left(\sum_{n=1}^{N} c_ne_n(t)+\zeta(t)\right).
	\end{array}
\end{equation}
Using the Young inequality, we obtain the bound 
\begin{equation}\label{eq:ZetaEstBoundartAct}
	\begin{array}{lll}
		\zeta^2(t) &\leq \left\|c \right\|_{N}^2\sum_{n=N+1}^{\infty}z_n^2(t).
	\end{array}
\end{equation}
Denote
\begin{equation}\label{eq:ErrDefNonDelayed0}
	\begin{array}{lllllll}
		&e^{N_0}(t)=\operatorname{col}\left\{e_n(t)\right\}_{n=1}^{N_0},\ e^{N-N_0}(t)=\operatorname{col}\left\{e_n(t)\right\}_{n=N_0+1}^{N},\\
		&\hat{z}^{N-N_0}(t)=\operatorname{col}\left\{\hat{z}_{n}(t)\right\}_{n=N_0+1}^N ,\  \mathcal{L}_0= \text{col}\left\{L_0,-L_0\right\},\\
		& \mathcal{K}_0 = \begin{bmatrix} K_0,&0_{1\times(N_0+1)}\end{bmatrix},\  A_1 = \operatorname{diag}\left\{-\lambda_{i}+q \right\}_{i=N_0+1}^N,\\
		& B_1 = \left[b_{N_0+1},\dots,b_N \right]^T,\ C_1=\left[c_{N_0+1}, \dots, c_N\right],
	\end{array}
\end{equation}
and
\begin{equation}\label{eq:ErrDefNonDelayed00}
	\begin{array}{lllllll}
		&\hspace{-3mm} F_0 = \scriptsize\begin{bmatrix}A_0+B_0K_0 & L_0C_0 \\ 0 & A_0-L_0C_0 \end{bmatrix},\  X_0(t) =\scriptsize \begin{bmatrix}
			\hat{z}^{N_0}(t)\\ e^{N_0}(t)
		\end{bmatrix}.
	\end{array}
\end{equation}
From \eqref{eq:zSolDecomp}, \eqref{eq:obsODENonDelayed}, \eqref{eq:C0A0}, \eqref{eq:ContDef}, \eqref{eq:en0}, \eqref{eq:ErrDefNonDelayed0} and \eqref{eq:ErrDefNonDelayed00} we observe that $e^{N-N_0}(t)$ satisfies
\begin{equation}\label{eNN0}
	\dot e^{N-N_0}(t)=A_1e^{N-N_0}(t)
\end{equation}
and is exponentially decaying, whereas the {\it reduced-order}  closed-loop system
\begin{equation}\label{eq:ClosedLoopReduced}
	\begin{aligned}
		&\dot{X}_0(t) = F_0X_0(t)+\mathcal{L}_0C_1e^{N-N_0}(t)+\mathcal{L}_0\zeta(t),\\
		& \dot{z}_n(t) = (-\lambda_n+q)z_n(t) +b_n\mathcal{K}_0X_0(t), \ n>N.
	\end{aligned}
\end{equation}
with $\zeta(t)$ subject to \eqref{eq:ZetaEstBoundartAct}
does not depend on $\hat{z}^{N-N_0}(t)$. Moreover, $\hat{z}^{N-N_0}(t)$ satisfies
\begin{equation}\label{eq:ztail}
	\dot {\hat{z}}^{N-N_0}(t)= A_1 \hat{z}^{N-N_0}(t)+B_1\mathcal{K}_0X_0(t)
\end{equation}
and is exponentially decaying provided $X_0(t)$ is exponentially decaying. Therefore, for stability of \eqref{eq:PDE1PointActNonDelayed} under the control law
\eqref{eq:ContDef} it is sufficient to show stability of the {reduced-order system} \eqref{eq:ClosedLoopReduced}. The latter can be considered as a {\it singularly perturbed system} with the slow state $X_0(t)$ and the fast infinite-dimensional state $z_n(t),\ n>N$.

Note that in \cite{RamiContructiveFiniteDim}, the full-order closed-loop system with the states $X_0, \hat z^{N-N_0}, e^{N-N_0}, z_n \ \ (n>N)$ was considered,
leading to full-order LMI conditions for stability.
In the present paper we derive stability conditions for the reduced-order  system \eqref{eq:ClosedLoopReduced} in terms of reduced-order LMI (see \eqref{eq:ReducedLMIs} below) for finding $N$ and the exponential decay rate $\delta$.
Differently from \cite{RamiContructiveFiniteDim}, the dimension of this LMI  will not grow with $N$.
Its feasibility for large $N$ will follow directly from the application of Schur complements. Moreover, if this LMI is feasible for $N$, it will be feasible for $N+1$. {\color{blue}To prove the exponential $L^2$-stability
	of the closed-loop system we employ the  Lyapunov function}
\begin{equation}\label{eq:PointActVNonDelayed_L2}
	\begin{array}{lll}
		&V(t)=V_0(t)+p_e\left|e^{N-N_0}(t) \right|^2,\\
		&V_0(t) = \left|X_0(t) \right|^2_{P_0}+\sum_{n=N+1}^{\infty}z^2_n(t)
	\end{array}
\end{equation}
where $0<P_0\in \mathbb{R}^{(2N_0+2)\times (2N_0+2)}$ and $0<p_e\in \mathbb{R}$. Note that $V(t)$ allows to compensate $\zeta(t)$ using \eqref{eq:ZetaEstBoundartAct}, whereas $V_0$ corresponds to \eqref{eq:ClosedLoopReduced} with $e^{N-N_0}=0$. 
\vspace{-0.3cm}
\begin{theorem}\label{Thm:PointActNonDelayed}
	Consider \eqref{eq:PDE1PointActNonDelayed} with measurement \eqref{eq:BoundMeasNonDelayed} where $c\in L^2(0,1)$ satisfies \eqref{eq:AsscnNonDelayed} and $z(\cdot,0)\in L^2(0,1)$. Let the control law be given by \eqref{eq:ContDef}. Let $\delta>0$ be a desired decay rate, $N_0\in \mathbb{Z}_+$ satisfy \eqref{eq:N0} and
	$N\ge  N_0+1$. Assume that $L_0$ and $K_0$ are obtained using \eqref{eq:GainsDesignL}  and \eqref{eq:GainsDesignK}, respectively. Let there exist $0<P_0\in \mathbb{R}^{(2N_0+2)\times (2N_0+2)}$ and a scalar $\alpha>0$ such that the following 
	LMI holds:
	\begin{equation}\label{eq:ReducedLMIs}
		\begin{array}{lll}
			&\scriptsize\begin{bmatrix}\Phi_{0} \ & P_0\mathcal{L}_0 \ & 0\\
				* \ & -2\left(\lambda_{N+1}-q-\delta\right)\left\|c\right\|_N^{-2} \ & 1\\
				* & * & -\frac{\alpha \left\|c\right\|_N^{2}}{\lambda_{N+1}}\end{bmatrix}<0,\\
			& \Phi_{0} = P_0F_0+F_0^TP_0+2\delta P_0+\frac{2\alpha}{\pi^2N} \mathcal{K}_0^T \mathcal{K}_0.
		\end{array}
	\end{equation}
	Then the solution $z(x,t)$ of \eqref{eq:PDE1PointActNonDelayed} subject to the control law  \eqref{eq:ContDef} and the corresponding observer $\hat{z}(x,t)$ given by \eqref{eq:ZhatSeries0}, \eqref{eq:obsODENonDelayed} satisfy the following inequalities:
	\begin{equation}\label{eq:L2Stability}
		\begin{array}{lll}
			\left\|z(\cdot,t) \right\|
			+\left\|z(\cdot,t)-\hat{z}(\cdot,t) \right\|\leq Me^{-\delta t}\left\|z(\cdot,0) \right\|
		\end{array}
	\end{equation}
	for some constant $M\geq 1$. Moreover, LMI \eqref{eq:ReducedLMIs} is always feasible if $N$ is large enough and feasibility of \eqref{eq:ReducedLMIs} for $N$ implies its feasibility for $N+1$.
\end{theorem}
\vspace{-0.3cm}
\begin{proof}
	Differentiating $V_0(t)$ along \eqref{eq:ClosedLoopReduced} we obtain
	\begin{equation}\label{eq:PointActStabAnalysisNonDelayed_NeummanBoundary}
		\begin{array}{lll}
			\dot{V}_0+2\delta V_0 =  X_0^T(t)\left[P_0F_0 +F_0^TP_0+2\delta P_0\right]X_0(t)\\
			+2X_0^T(t)P_0\mathcal{L}_0\zeta(t) +2\sum_{n=N+1}^{\infty}(-\lambda_n+q+\delta)z_n^2(t)\\
			+2\sum_{n=N+1}^{\infty}z_n(t)b_n\mathcal{K}_0X_0(t)+2X_0^T(t)P_0\mathcal{L}_0C_1e^{N-N_0}(t).
		\end{array}
	\end{equation}	
	The Young inequality implies
	\begin{equation}\label{eq:PointActCrosTermNonDelayed}
		\begin{aligned}
			&2\sum_{n=N+1}^{\infty} z_n(t)b_n\mathcal{K}_0X_0(t) = 2\sum_{n=N+1}^{\infty} \lambda_n^{\frac{1}{2}}z_n(t) \frac{b_n}{\lambda_n^\frac{1}{2}}\mathcal{K}_0X_0(t) \\
			&\overset{\eqref{eq:IntegTest}}{\leq}\frac{1}{\alpha}\sum_{n=N+1}^{\infty} \lambda_n z_n^2(t) + \frac{2\alpha}{\pi^2N}  \left|\mathcal{K}_0X_0(t) \right|^2,
		\end{aligned}
	\end{equation}
	where $\alpha>0$. From monotonicity of $\lambda_n, \ n\in \mathbb{Z}_+$ we have
	\begin{equation}\label{eq:zetaCompensate0}
		\begin{array}{lll}
			&2\sum_{n=N+1}^{\infty}\left(-\lambda_n+q+\delta+\frac{1}{2\alpha}\lambda_n \right)z_n^2(t)\\
			&\overset{\eqref{eq:ZetaEstBoundartAct}}{\leq}2\left(-\lambda_{N+1}+q+\delta+\frac{1}{2\alpha}\lambda_{N+1} \right)\left\|c \right\|_{N}^{-2}\zeta^2(t)
		\end{array}
	\end{equation}
	provided $-\lambda_{N+1}+q+\delta+\frac{1}{2\alpha}\lambda_{N+1}\leq 0$. Differentiating $p_e\left|e^{N-N_0}(t)\right|^2$ we have
	\begin{equation}\label{eq:LastModes}
		\begin{array}{lll}
			&\frac{d}{dt}\left[p_e\left|e^{N-N_0}(t) \right|^2 \right]+2\delta p_e\left|e^{N-N_0}(t) \right|^2\\
			&\hspace{5mm}= 2p_e\left(e^{N-N_0}(t) \right)^T\left(A_1+\delta I \right)e^{N-N_0}(t).
		\end{array}
	\end{equation}
	Let $\eta(t) = \text{col}\left\{X_0(t),\zeta(t),e^{N-N_0}(t) \right\}$. From \eqref{eq:PointActStabAnalysisNonDelayed_NeummanBoundary}-\eqref{eq:LastModes}
	\begin{equation}\label{eq:PointActStabResultNonDelayed}
		\begin{aligned}
			&\dot{V}+2\delta V \leq \eta^T(t)\Psi \eta(t)\leq 0
		\end{aligned}
	\end{equation}
	if
	\begin{equation}\label{eq:PointActLMIsNonDelayed}
		\begin{aligned}
			&\Psi=\scriptsize\begin{bmatrix}\Omega_1 \ & \Omega_2\\
				* \ & 2p_e\left(A_1+\delta I \right) \end{bmatrix}<0,\ \Omega_2 = \scriptsize\begin{bmatrix} P_0\mathcal{L}_0C_1\\ 0 \end{bmatrix},\\
			&\Omega_1 = \scriptsize\begin{bmatrix}
				\Phi_{0} \ & P_0\mathcal{L}_0\\
				* \ & -2\left(\lambda_{N+1}-q-\delta-\frac{1}{2\alpha}\lambda_{N+1} \right)\left\|c \right\|_{N}^{-2}
			\end{bmatrix}.
		\end{aligned}
	\end{equation}
	{Note that $A_1+\delta I<0$ by \eqref{eq:N0}. Therefore, by Schur complement, $\Psi<0$ iff
		\begin{equation}\label{Omega1}
			\begin{array}{lll}
				&\Omega_1 - \frac{1}{2p_e}P_0\mathcal{L}_0C_1\left(A_1+\delta I \right)^{-1}C_1^T\mathcal{L}_0^TP_0<0
			\end{array}
		\end{equation}
		Taking  $p_e\to \infty$ in \eqref{Omega1} ($p_e$  does not appear in $\Omega_1$), we find that $\Psi<0$ iff $\Omega_1<0$ and the latter is equivalent, by Schur complement, to \eqref{eq:ReducedLMIs}.
		{\color{blue} Thus, \eqref{eq:ReducedLMIs} guarantees \eqref{eq:PointActStabResultNonDelayed} implying the exponential stability of the closed-loop system \eqref{eNN0}-\eqref{eq:ztail} and \eqref{eq:L2Stability}.
			
			To prove the feasibility of \eqref{eq:ReducedLMIs} for large $N$}, choose $\alpha = 1 $ and $N_1\in \mathbb{N}$ such that for $N\geq N_1$, we have  $\Phi_0<0$ in \eqref{eq:ReducedLMIs} for some $P_0>0$. This is possible since $\left\|\mathcal{K}_0 \right\|$ is independent of $N$ and $F_0$ is Hurwitz (see \eqref{eq:GainsDesignL}, \eqref{eq:GainsDesignK} and \eqref{eq:ErrDefNonDelayed00}). By increasing $N_1$ we can also assume that for $N\geq N_1$ we have $\frac{1}{2}\lambda_{N+1}-q-\delta>0$. Then, by Schur complement $\Omega_1<0$ holds iff
		\begin{equation}\label{eq:LMIFin}
			\Phi_0+\frac{\left\|c\right\|_N^2}{\lambda_{N+1}-2q-2\delta}P_0\mathcal{L}_0\mathcal{L}_0^TP_0<0.
		\end{equation}
		Since $\left\|\mathcal{L}_0\right\|$ is independent of $N$, $\lambda_{N+1}\overset{N\to \infty}{\longrightarrow}\infty$ and $\left\|c\right\|_N^2\overset{N\to \infty}{\longrightarrow}0$, by increasing $N_1$ if needed, \eqref{eq:LMIFin} holds. Finally, note that by replacing $N$ with $N+1$ in \eqref{eq:LMIFin}, the positive terms on the left-hand side decrease, whereas $P_0F_0+F_0^TP_0+2\delta P_0$ is unchanged. This shows that feasibility for $N$ implies feasibility for $N+1$.}	
\end{proof}
\begin{remark} \label{rem_other PDEs}
	\textcolor{blue}{
		The reduced-order LMIs can be derived similarly 
		for other parabolic PDEs (including heat equations with variable diffusion and reaction coefficients as  in \cite{RamiContructiveFiniteDim} and Kuramoto-Sivashinsky equation (KSE) as in \cite{Rami_CDC20}):
		for
		the reduced-order closed-loop system (without $\hat z^{N-N_0}$)  the Lyapunov function of the form 
		$V(t)=V_0(t)+p_e|e^{N-N_0}(t)|^2$
		should be employed, where $p_e>0$ is large and $V_0$ corresponds to the reduced-order closed-loop system with the omitted $e^{N-N_0}$. 
		Then for $p_e\to \infty$ the reduced-order LMI 
		will be obtained.
		Moreover, 
		it can be shown that for the mentioned above PDEs 
		the similar controller under Neumann actuation and non-local measurement  leads to $L^2$ convergence without dynamic extension.
		This allows treating  fast-varying input/output delays as presented in Section \ref{sec_delay}. 
	}
	%
	%
\end{remark}

\section{Delayed $L^2$-stabilization}\label{sec_delay}
We consider the delayed reaction-diffusion system
\begin{equation}\label{eq:PDEPointActIntervalDelayed}
	\begin{aligned}
		& z_t(x,t)=z_{xx}(x,t)+qz(x,t),\\
		& z_x(0,t)=0, \quad z_x(1,t)=u(t-\tau_u(t)),
	\end{aligned}
\end{equation}
under delayed Neumann actuation and delayed non-local measurement
\begin{equation}\label{eq:BoundMeas}
	y(t) = \left<z(\cdot,t-\tau_y(t)) ,c\right>, \quad c\in L^2(0,1).
\end{equation}
\textcolor{blue}{Here $z(\cdot,t-\tau_y(t))=z(\cdot,0)$ for $t-\tau_y(t)\leq 0$ and $\tau_y(t)\geq 0$ is a \emph{known} continuously differentiable output delay with locally Lipschitz derivative from the interval}
\begin{equation}\label{eq:Delay bounds}
	0<\tau_m\leq \tau_y(t)\leq \tau_M.
\end{equation}
The lower bound on $\tau_y(t)$ is required for well-posedness only.
The continuously differentiable input delay $\tau_u(t)$ belongs to the known interval
\begin{equation}\label{eq:tauUdecomp}
	\tau_u(t)\in [r, r+\theta_M], \quad t\geq 0
\end{equation}
where $r>0$ and has locally Lipschitz derivative. 
Henceforth the dependence of $\tau_y(t)$ and $\tau_u(t)$ 
on $t$ will be suppressed to shorten notations.

We present the solution of \eqref{eq:PDEPointActIntervalDelayed} as \eqref{eq:zSolDecomp}. {\color{blue} Then \eqref{eq:ZOdesPointActNonDelayed} has the form}
\begin{equation}\label{eq:ZOdesPointActIntervalDelayed}
	\begin{aligned}
		&\dot{z}_n(t) = (- \lambda_n + q) z_n(t) + b_n u(t-\tau_u) \\
		&b_0 = 1,\ \ b_n = {(-1)}^n \sqrt{2}, \ \ n=0,1,\dots.
	\end{aligned}
\end{equation}
Let $\delta>0$. There exists some $N_0 \in \mathbb{Z}_+$ such that \eqref{eq:N0} holds. $N_0$ will define the dimension of the controller, whereas $N\ge N_0+1$ will be the dimension of the observer. To derive stability conditions in terms of the reduced-order LMIs, in Sections \ref{sec_robust} and \ref{sec_predictor} we consider the case of known input delay and construct a $N$-dimensional observer of the form \eqref{eq:ZhatSeries0}, where $\hat{z}_n(t)$ satisfy the ODEs
\begin{equation}\label{eq:obsODEDelayednew}
	\begin{aligned}
		\dot{\hat{z}}_n(t) &= (-\lambda_n+q)\hat{z}_n(t) + b_nu(t-\tau_u)\\
		&-l_n\left[\left<\hat{z}(\cdot,t-\tau_y) ,c\right>- y(t)\right],\quad t\ge 0,\\
		\hat{z}_n(t)&=0,\quad t\le 0, \qquad 0\leq n\leq N.
	\end{aligned}
\end{equation}
Here $l_n \ (0\leq n\leq N)$ are scalars and $l_{N_0+1}=...=l_N=0$.
In Section \ref{Sec:PredictorKnownDelay} we consider unknown $\tau_u$, where $u(t-\tau_u)$ in the observer equation \eqref{eq:obsODEDelayednew} is replaced by $u(t-r)$.

Recall the notations \eqref{eq:C0A0}. Under the assumption \eqref{eq:AsscnNonDelayed}, $(A_0,C_0)$ is observable. Let $L_0 = [l_0,\dots, l_{N_0}]^T$ satisfy the Lyapunov inequality \eqref{eq:GainsDesignL} for some $0<P_{\text{o}}\in \mathbb{R}^{(N_0+1)\times (N_0+1)}$.  Similarly, \eqref{eq:bnNeq0} implies that $(A_0,B_0)$ is controllable. Let $K_0\in \mathbb{R}^{1\times (N_0+1)}$ satisfy \eqref{eq:GainsDesignK}
for some $0<P_{\text{c}}\in \mathbb{R}^{(N_0+1)\times (N_0+1)}$.

\subsection{Stabilization robust with respect to delays}\label{sec_robust}

We propose the control law \eqref{eq:ContDef},
	which is based on the $N$-dimensional observer \eqref{eq:ZhatSeries0}, \eqref{eq:obsODEDelayednew}.
	
	\textcolor{blue}{We show well-posedness of the closed-loop system \eqref{eq:PDEPointActIntervalDelayed}, \eqref{eq:obsODEDelayednew} subject to the measurement \eqref{eq:BoundMeas} and control input \eqref{eq:ContDef}. Note that well-posedness of the closed-loop systems in Sections 3.2 and 3.3 can be proved similarly and it is omitted for brevity. We assume that there exist \emph{unique} $t_y^*\in [\tau_m,\tau_M]$ and $t_u^*\in [r,r+\theta_M]$ such that $t_y^*-\tau_y(t_y^*)=t_u^*-\tau_u(t_u^*)=0$ (see Figure \ref{fig:line}). Recall \eqref{eq:CalADef} and let $z(\cdot,0)\in \mathcal{D}\left(\mathfrak{A} \right)$. We use the step method for well-posedness.}
	\begin{figure}
		\centering
		\includegraphics[width=85mm,scale=1]{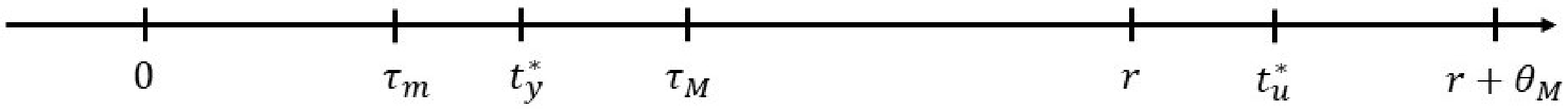}
		\caption{\textcolor{blue}{Well-posedness: time instances for the step method}
	}\label{fig:line}
\end{figure}

\textcolor{blue}{In the first step, consider $t\in [0,t_u^*]$. By \eqref{eq:ContDef} and \eqref{eq:obsODEDelayednew} we have $u(t-\tau_u(t))\equiv 0, \ t\in [0,t_u^*]$. Since $z(\cdot,0)\in \mathcal{D}\left(\mathfrak{A} \right)$, by Theorem 4.1.3 in \cite{pazy1983semigroups}, the PDE \eqref{eq:PDEPointActIntervalDelayed} has a unique  solution in $C^1([0,t_u^*],L^2(0,1))$. Next, consider the ODEs \eqref{eq:obsODEDelayednew}. First, let $t\in [0,t_y^*]$. By assumption we have $y(t) \equiv \left<z(\cdot,0),c \right>$ and $\hat{z}(\cdot,t-\tau_y(t))\equiv 0$. Hence, \eqref{eq:obsODEDelayednew} have a unique solution that is continuously differentiable with Lipschitz derivative on $[0,t_y^*]$. Next, let $t\in [t_y^*,t_y^*+\tau_m]$ and consider \eqref{eq:obsODEDelayednew} with initial condition $\left\{ \hat{z}_n(t_y^*)\right\}_{n=0}^N$ obtained at the previous step. Note that by assumption we have $0\leq t-\tau_y(t)\leq t_y^*, \ \forall t\in [t_y^*,t_y^*+\tau_m]$. By the previous results, the last two terms in the ODEs \eqref{eq:obsODEDelayednew} (thought of as non-homogeneous terms) are Lipschitz continuous on $t\in [t_y^*,t_y^*+\tau_m]$. Hence, there exists a unique solution that is continuously differentiable with Lipschitz derivative for $t\in [t_y^*,t_y^*+\tau_m]$. Gluing the solutions together we have a unique solution that is continuously differentiable with Lipschitz derivative for $t\in [0,t_y^*+\tau_m]$. Repeating the same arguments step by step on $[t_y^*+\tau_m,t_y^*+2\tau_m],[t_y^*+2\tau_m*,t_y^*+3\tau_m],\dots$ until $t=\tau_u^*$, we conclude that  \eqref{eq:obsODEDelayednew} has a unique solution that is continuously differentiable with Lipschitz derivative for $t\in [0,t_u^*]$.}

\textcolor{blue}{In the second step, consider $t\in[t_u^*,t_u^*+\tau_m]$. Here, the control input is no longer identically zero.  We have $t-\tau_y(t)\in [0,t_u^*]$ and $t-\tau_u(t)\in [0,t_u^*)$ for $ t\in[t_u^*,t_u^*+\tau_m]$. Consider first \eqref{eq:PDEPointActIntervalDelayed} with initial condition $z(\cdot,t_u^*)\in \mathcal{D}\left(\mathfrak{A} \right)$ obtained in the previous step. Introducing $w(x,t) = z(x,t)-\frac{1}{2}x^2u(t-\tau_u(t))$ we have the following equivalent PDE
	\begin{equation*}
		\begin{array}{lll}
			& w_t(x,t)=w_{xx}(x,t)+qw(x,t)\\
			&+\left(\frac{qx^2}{2}+1\right)u(t-\tau_u(t))-\frac{x^2}{2}\dot{u}(t-\tau_u(t))(1-\dot{\tau}_u(t)),\\
			& w_x(0,t)=0, \quad w_x(1,t)=0.
		\end{array}
	\end{equation*}
	with initial condition $w(\cdot,t_u^*)=z(\cdot,t_u^*)$. By results of the previous step, $\eqref{eq:ContDef}$ and the assumption on $\tau_u(t)$, the last two terms on the right-hand-side are Lipschitz continuous  non homogeneities on $t\in[t_u^*,t_u^*+\tau_m]$.  By Theorems  6.3.1 and 6.3.3 in \cite{pazy1983semigroups} (see similar arguments in  (2.15)-(2.18)) 
	we obtain a unique classical solution such that $z\in C([t_u^*,t_u^*+\tau_m],L^2(0,1))\cap C^1((t_u^*,t_u^*+\tau_m],L^2(0,1))$ and $z(\cdot,t)\in H^2(0,1)$ with $z_x(0,t)=0, \ z_x(1,t)=u(t-\tau_u(t))$ for $t\in[t_u^*,t_u^*+\tau_m]$. Next, consider \eqref{eq:obsODEDelayednew} for $t\in [t_u^*,t_u^*+\tau_m]$ with initial condition $\left\{ \hat{z}_n(t_u^*)\right\}_{n=0}^N$ obtained at the previous step. Since $t-\tau_y(t)\in [0,t_u^*]$ and $t-\tau_u(t)\in [0,t_u^*)$ for $ t\in[t_u^*,t_u^*+\tau_m]$, the three last terms in the ODEs are Lipschitz continuous on $t\in[t_u^*,t_u^*+\tau_m]$. Hence, \eqref{eq:obsODEDelayednew} has a unique solution that is continuously differentiable with Lipschitz derivative on $t\in[t_u^*,t_u^*+\tau_m]$. Continuing step-by-step on $[t_u^*+\tau_m*,t_u^*+2\tau_m],[t_u^*+2\tau_m*,t_u^*+3\tau_m],\dots$ we obtain the existence of a unique classical solution $z\in C([0,\infty),L^2(0,1))\cap C^1((0,\infty)\setminus S,L^2(0,1))$, where $S=\left\{\tau_u^*+j\tau_m\right\}_{j=0}^{\infty}$. Moreover, $z(\cdot,t)\in H^2(0,1)$ with $z_x(0,t)=0, \ z_x(1,t)=u(t-\tau_u(t))$ for $t\in[0,\infty)$.}

Recall the estimation error given in \eqref{eq:ErrDefNonDelayed0}. The last term on the right-hand side of \eqref{eq:obsODEDelayednew} can be written as
\begin{equation}\label{eq:IntroZetaDelayed}
	\begin{array}{ll}
		\left<\hat{z}(\cdot,t-\tau_y) ,c\right>- y(t) =-\sum_{n=0}^{N} c_ne_n(t-\tau_y)-\zeta(t-\tau_y)
	\end{array}
\end{equation}
with $\zeta(t)$ given in \eqref{eq:IntroZetaNonDelayed} and satisfies \eqref{eq:ZetaEstBoundartAct}. Then the error equations for $t\geq 0$ and $0\leq n\leq N_0$ are
\begin{equation}\label{eq:en}
	\begin{array}{ll}
		&\dot e_n(t)=(-\lambda_n+q)e_n(t)\\
		&\hspace{5mm}-l_n\left(\sum_{n=1}^{N} c_ne_n(t-\tau_y)+\zeta(t-\tau_y)\right),\\
		&e_n(t)=\left<z_0,\phi_n \right>, \quad t\leq 0.
	\end{array}
\end{equation}
Recall the notations \eqref{eq:C0A0}, \eqref{eq:ErrDefNonDelayed0} and \eqref{eq:ErrDefNonDelayed00} and let
\begin{equation}\label{eq:notationsDelay}
	\begin{array}{lll}
		&\mathcal{B}_0= \text{col}\left\{B_0,0_{(N_0+1)\times 1}\right\}, \ \mathcal{C}_0 = [0_{1\times(N_0+1)},C_0],\\
		&\Upsilon_{y}(t) = X_0(t-\tau_y)-X_0(t),\ \Upsilon_r(t) = X_0(t-r)-X_0(t),\\
		&\Upsilon_u(t) = X_0(t-\tau_u)-X_0(t-r).
	\end{array}
\end{equation}
As in the non-delayed case, here $e^{N-N_0}(t)=e^{A_1t}e(0)$ satisfies  \eqref{eNN0}. Substituting $e^{N-N_0}(t-\tau_y)=e^{-A_1\tau_y}e^{N-N_0}(t)$,
the reduced-order (i.e decoupled from $\hat{z}^{N-N_0}(t)$) closed-loop system  is governed by
\begin{equation}\label{eq:X0Delay}
	\begin{array}{llllll}
		\dot{X}_0(t) = & F_0X_0(t)+\mathcal{B}_0\mathcal{K}_0\left [\Upsilon_u(t)+
		\Upsilon_r(t)\right]+\mathcal{L}_0\mathcal{C}_0\Upsilon_y(t)\\
		& 
		+\mathcal{L}_0\zeta(t-\tau_y)+\mathcal{L}_0C_1e^{-A_1\tau_y}e^{N-N_0}(t),\\
		\dot{z}_n(t)=& (-\lambda_n+q)z_n(t)+b_n\mathcal{K}_0X_0(t)\\
		&+b_n\mathcal{K}_0\left [\Upsilon_u(t)+
		\Upsilon_r(t)\right], \quad n>N,
	\end{array}
\end{equation}
with $\zeta(t)$ subject to \eqref{eq:ZetaEstBoundartAct},
where $e^{N-N_0}(t)$ is an exponentially decaying input.
Note that $\hat{z}^{N-N_0}(t)$  satisfies
\begin{equation}\label{eq:ztailDelay}
	\dot {\hat{z}}^{N-N_0}(t)= A_1 \hat{z}^{N-N_0}(t)+B_1\mathcal{K}_0X_0(t-\tau_u)
\end{equation}
and is exponentially decaying provided $X_0(t)$ is exponentially decaying. For $L^2$-stability analysis of \eqref{eq:X0Delay}, \eqref{eNN0} we fix $\delta_0>\delta$ and define the Lyapunov functional
\begin{equation}\label{eq:VComponents0}
	W(t):=V(t)+\sum_{i=0}^2 V_{S_i}(t)+\sum_{i=0}^2 V_{R_i}(t),
\end{equation}
where $V(t)$ is given by \eqref{eq:PointActVNonDelayed_L2} and
\begin{equation}\label{eq:VComponentsR}
	\begin{array}{llll}
		&\hspace{-3mm}V_{S_0}(t):=\int_{t-r}^{t} e^{-2\delta_0(t-s)}\left|\mathcal{K}_0X_0(s)\right|^2_{S_0} ds,\\
		&\hspace{-3mm}V_{R_0}(t):= r\int_{-r}^{0} \int_{t+\theta}^t e^{-2\delta_0(t-s)} \left|\mathcal{K}_0\dot{X}_0(s)\right|^2_{R_0}ds d\theta \\
		&\hspace{-3mm}V_{S_1}(t):=\int_{t-r-\theta_M}^{t-r} e^{-2\delta_0(t-s)}\left|\mathcal{K}_0X_0(s)\right|^2_{S_1} ds,\\
		&\hspace{-3mm}V_{R_1}(t):= \theta_M\int_{-r-\theta_M}^{-r
		} \int_{t+\theta}^t e^{-2\delta_0(t-s)} \left|\mathcal{K}_0 \dot{X}_0(s)\right|^2_{R_1}ds d\theta \\
		&\hspace{-3mm}V_{S_2}(t):=\int_{t-\tau_M}^t e^{-2\delta_0(t-s)}\left|X_0(s)\right|^2_{S_2}ds,\\
		&\hspace{-3mm}V_{R_2}(t):= \tau_M\int_{-\tau_M}^0 \int_{t+\theta}^t e^{-2\delta_0(t-s)} \left|\dot{X}_0(s)\right|^2_{R_2}ds d\theta. \\
	\end{array}
\end{equation}
Here $S_2,R_2>0$ are square matrices of order $2N_0+2$  and $S_0, R_0, S_1,R_1>0$ are scalars. $V_{S_0}$ and $V_{R_0}$ are introduced to compensate $\Upsilon_r(t)$. $V_{S_1}$ and $V_{R_1}$ are used to compensate $\Upsilon_{u}(t)$. $V_{S_2}$ and $V_{R_2}$ are used to compensate $\Upsilon_y(t)$. Finally, to compensate $\zeta(t-\tau_y)$  we will use Halanay's inequality:
\begin{lemma}\label{Lem:Halanay}
	(Halanay's inequality).
	\newline
	Let $0<\delta_1<\delta_0$ and let $W:[t_0-\tau_M,\infty)\longrightarrow [0,\infty) $ be an absolutely continuous function
	that satisfies
	\begin{equation*}
		\dot{W}(t) + 2 \delta_0 W (t) - 2 \delta_1 \sup_{-\tau_M\leq \theta \leq 0} W(t+\theta) \leq 0, \quad t \geq t_0.
	\end{equation*}
	Then $
	W(t) \leq \exp \big(-2\delta_{\tau_M}(t-t_0) \big) \sup_{-\tau_M\leq \theta \leq 0} W(t_0+\theta), \ t \geq t_0,$ 
where $ \delta_{\tau_M}>0 $ is a unique positive solution of
\begin{equation}\label{eq:DecRate}
	\delta_{\tau_M}=\delta_0-\delta_1 \exp(2 \delta_{\tau_M} \tau_M).
\end{equation}
\end{lemma}
\textcolor{blue}{To state the main result of this section, we employ the following notations for $G_1\in \mathbb{R}$ and $G_2\in \mathbb{R}^{2(N_0+1)\times2(N_0+1)}$ and $0<\alpha,\alpha_1,\alpha_2\in \mathbb{R}$:}
\begin{equation}\label{eq:PsiMatrix2}
	\begin{array}{lll}
		&\hspace{-3mm}\Psi_{0} = \scriptsize \left[
		\begin{array}{c|c}
			\Theta & \Sigma_1 \ \ \Sigma_2  \\
			\hline
			* & \operatorname{diag}\left\{\Gamma_1, \Gamma_2 \right\}
		\end{array}
		\right], \ \Theta = \scriptsize\begin{bmatrix}\Phi_{\text{delay}} & P_0\mathcal{L}_0\\
			* & -2\delta_1\left\|c \right\|_N^{-2} \end{bmatrix},\\
		&\hspace{-3mm}\Sigma_1 =\scriptsize\begin{bmatrix}
			P_0\mathcal{L}_0\mathcal{C}_0-2\delta_1P_0-\varepsilon_MS_2& -\varepsilon_MS_2\\0 & 0
		\end{bmatrix},\\
		&\hspace{-3mm}\Sigma_2 = \scriptsize \begin{bmatrix}
			P_0\mathcal{B}_0-\varepsilon_{r,M}\mathcal{K}_0^TS_1 \ & \Xi_1  \ & -\varepsilon_{r,M}\mathcal{K}_0^TS_1\\ 0 \ & 0 \ & 0
		\end{bmatrix},\\
		&\hspace{-3mm}\Gamma_1 = \scriptsize\begin{bmatrix}
			-2\delta_1P_0-\varepsilon_M(R_2+S_2) \ & -\varepsilon_M(S_2+G_2)\\* \ & -\varepsilon_M(R_2+S_2)
		\end{bmatrix},\\
		&\hspace{-3mm}\Gamma_2 = \scriptsize\begin{bmatrix}
			-\varepsilon_{r,M}(R_1+S_1)+\frac{2\alpha_1}{\pi^2N}\ & -\varepsilon_{r,M}S_1\ & -\varepsilon_{r,M}(S_1+G_1)\\
			* \ & \Xi_2 \ &-\varepsilon_{r,M}S_1\\
			* \ & * \ & -\varepsilon_{r,M}(R_1+S_1)
		\end{bmatrix},\\
		&\hspace{-3mm}\Phi_{\text{delay}} = \Phi_{0}+(1-\varepsilon_r)\mathcal{K}_0^TS_0\mathcal{K}_0,\\
		&\hspace{7mm}+(\varepsilon_r-\varepsilon_{r,M})\mathcal{K}_0^TS_1\mathcal{K}_0+ (1-\varepsilon_M)S_2,
	\end{array}
\end{equation}
\begin{equation*}
	\begin{array}{lll}	
		&\hspace{-3mm}\Xi_1 = P_0\mathcal{B}_0-\varepsilon_{r}\mathcal{K}_0^TS_0+(\varepsilon_r - \varepsilon_{r,M})\mathcal{K}_0^TS_1,	\\	
		&\hspace{-3mm}\Xi_2 = \frac{2\alpha_2}{\pi^2N}-\varepsilon_r(R_0+S_0)+(\varepsilon_r - \varepsilon_{r,M})S_1,\\
		&\hspace{-3mm} \Lambda_{0} = [F_0, \mathcal{L}_0, \mathcal{L}_0\mathcal{C}_0, 0 ,\mathcal{B}_0,\mathcal{B}_0,0],\\
		&\hspace{-3mm} \varepsilon_{\tau} = e^{-2\delta_0 \tau}, \quad \tau\in \left\{r,\tau_M,r+\theta_M\right\}.
	\end{array}
\end{equation*}
\begin{theorem}\label{Thm:NonLocalL2Delay}
	\textcolor{blue}{Consider \eqref{eq:PDEPointActIntervalDelayed}, measurement \eqref{eq:BoundMeas} with $c\in L^2(0,1)$ satisfying \eqref{eq:AsscnNonDelayed}, control law \eqref{eq:ContDef}. Let $\delta_0>\delta>0$ and $\delta_1 = \delta_0-\delta$. Let $N_0\in \mathbb{Z}_+$ satisfy \eqref{eq:N0} and $N\geq N_0+1$. Assume that $L_0$ and $K_0$ are obtained using \eqref{eq:GainsDesignL}  and \eqref{eq:GainsDesignK}, respectively. Given $r,\theta_M,\tau_M>0$, let there exist positive definite matrices $P_0,S_2,R_2\in \mathbb{R}^{2(N_0+1)\times 2(N_0+1)}$, scalars $S_0,R_0,S_1,R_1, \alpha,\alpha_1,\alpha_2>0$, $G_1\in \mathbb{R}$ and $G_2\in \mathbb{R}^{2(N_0+1)\times 2(N_0+1)}$ such that
		\begin{equation}\label{eq:3.1LMI}
			\begin{array}{lll}
				&\scriptsize\begin{bmatrix}R_1 & G_1\\ * & R_1 \end{bmatrix}\geq 0, \ \ \scriptsize\begin{bmatrix}R_2 & G_2\\ * & R_2 \end{bmatrix}\geq 0,\\
				&\scriptsize\left[
				\begin{array}{c|c}
					-\lambda_{N+1}+q+\delta_0 & 1 \qquad  1 \qquad 1 \\
					\hline
					* & -\frac{2}{\lambda_{N+1}}\operatorname{diag}\left\{\alpha, \alpha_1, \alpha_2 \right\}
				\end{array}
				\right]< 0
			\end{array}
		\end{equation}
		and
		\begin{equation}\label{eq:PsiLMI}
			\Psi_{0} +\Lambda_{0}^T \left[\mathcal{K}_0^T\left(r^2R_0+\theta_M^2R_1\right)\mathcal{K}_0+\tau_M^2R_2 \right]\Lambda_{0}\normalsize<0
		\end{equation}
		hold. Then the solution $z(x,t)$ to \eqref{eq:PDEPointActIntervalDelayed} under the control law \eqref{eq:ContDef} and the 
		observer $\hat z(x,t)$ defined by \eqref{eq:ZhatSeries0}, \eqref{eq:obsODEDelayednew}
		satisfy 
		\begin{equation}\label{eq:L2StabilityDelay}
			\begin{aligned}
				\left\|z(\cdot,t)\right\|
				+\left\|z(\cdot,t)-\hat{z}(\cdot,t)\right\|\leq Me^{-\delta_{\tau_M} t}\left\|z(\cdot,0)\right\|
			\end{aligned}
		\end{equation}
		for some $M\geq 1$, where  $\delta_{\tau_M}>0$ is defined by \eqref{eq:DecRate}.  Moreover, LMIs \eqref{eq:3.1LMI}, \eqref{eq:PsiLMI} are always feasible for large enough $N$ and  small enough $\tau_M,\theta_M$ and $r$ and their feasibility for $N$ implies feasibility for $N+1$.}
\end{theorem}
\begin{proof}
	Differentiating $V(t)$ along \eqref{eNN0}, \eqref{eq:X0Delay} we obtain
	\begin{equation}\label{eq:VnomDelay}
		\begin{array}{lll}
			&\hspace{-4mm}\dot{V}+2\delta V =  X_0^T(t)\left[P_0F_0 +F_0^TP_0+2\delta_0 P_0\right]X_0(t)\\
			&\hspace{-4mm}+2X_0^T(t)P_0\mathcal{L}_0\zeta(t-\tau_y)+2X_0^T(t)P_0\mathcal{B}_0 \mathcal{K}_0\left[\Upsilon_u(t)+\Upsilon_r(t) \right]\\
			&\hspace{-4mm}+2X_0^T(t)P_0\mathcal{L}_0\mathcal{C}_0\Upsilon_y(t)+2X_0^T(t)P_0\mathcal{L}_0C_1e^{-A\tau_y}e^{N-N_0}(t)\\
			&\hspace{-4mm}+2\sum_{n=N+1}^{\infty}(-\lambda_n+q+\delta_0)z_n^2(t)+2p_e\left|e^{N-N_0}(t) \right|^2_{A_1+\delta_0 I}\\
			&\hspace{-4mm}+2\sum_{n=N+1}^{\infty}z_n(t)b_n\mathcal{K}_0\left[X_0(t)+\Upsilon_u(t)+\Upsilon_r(t) \right].
		\end{array}
	\end{equation}	
	By arguments similar to \eqref{eq:PointActCrosTermNonDelayed} we have
	\begin{equation}\label{eq:CrosTermNonDelayed}
		\begin{array}{lll}
			&2\sum_{n=N+1}^{\infty}z_n(t)b_n\mathcal{K}_0\left[X_0(t)+\Upsilon_u(t)+\Upsilon_r(t) \right]\\
			&\leq \left[\frac{1}{\alpha}+\frac{1}{\alpha_1}+\frac{1}{\alpha_2} \right]\sum_{n=N+1}^{\infty}\lambda_nz_n^2(t)+\frac{2\alpha}{\pi^2N}\left|\mathcal{K}_0X_0(t) \right|^2\\
			&+\frac{2\alpha_1}{\pi^2N}\left|\mathcal{K}_0\Upsilon_u(t) \right|^2+\frac{2\alpha_2}{\pi^2N}\left|\mathcal{K}_0\Upsilon_r(t) \right|^2.
		\end{array}
	\end{equation}
	Differentiation of  $V_{S_0}$ and $V_{R_0}$ leads to
	\begin{equation}\label{eq:AugFunc_2}
		\begin{array}{lll}
			&\dot{V}_{S_0}+2\delta_0V_{S_0} = \left|\mathcal{K}_0X_0(t) \right|^2_{S_0} \\
			&\hspace{8mm}-\varepsilon_r\left|\mathcal{K}_0X_0(t)+\mathcal{K}_0\Upsilon_r(t)\right|^2_{S_0},\\
			&\dot{V}_{R_0}+2\delta_0V_{R_0} =r^2\left|\mathcal{K}_0\dot{X}_0(t) \right|^2_{R_0}\\
			&\hspace{8mm}-r\int_{t-r}^te^{-2\delta_0(t-s)}\left|\mathcal{K}_0\dot{X}_0(s) \right|^2_{R_0}ds.
		\end{array}
	\end{equation}
	By using Jensen's inequality we have
	\begin{equation*}\label{eq:JensenR0}
		\hspace{-1mm}-r\int_{t-r}^te^{-2\delta_0(t-s)}\left|\mathcal{K}_0\dot{X}_0(s) \right|^2_{R_0}ds\leq -\varepsilon_r\left|\mathcal{K}_0\Upsilon_r(t) \right|^2_{R_0}.
	\end{equation*}
	Let
	\begin{equation}\label{eq:muDelays}
		\begin{array}{lll}
			&Q_u(t) = X_0(t-r-\theta_M)-X_0(t-\tau_u),\\
			&Q_y(t) = X_0(t-\tau_M)-X_0(t-\tau_y).
		\end{array}
	\end{equation}
	Differentiation of $V_{S_i}$ and $V_{R_i}$, $i\in \left\{1,2\right\}$, gives:
	\begin{equation*}\label{eq:AugFunc_1}
		\begin{array}{lll}
			&\dot{V}_{S_1}+2\delta_0V_{S_1} = \varepsilon_r\left|\mathcal{K}_0\Upsilon_r(t)+\mathcal{K}_0X_0(t) \right|^2_{S_1} \\
			&-\varepsilon_{r+\theta_M}\left|\mathcal{K}_0\left(Q_u(t)+\Upsilon_u(t)+\Upsilon_r(t)+X_0(t) \right)\right|^2_{S_1},\\
			&\dot{V}_{S_2}+2\delta_0V_{S_2} =\left|X_0(t) \right|_{S_2}^2\\
			&-\varepsilon_{\tau_M}\left|Q_y(t)+\Upsilon_y(t)+X_0(t) \right|^2_{S_2},\\
			&\dot{V}_{R_1}+2\delta_0V_{R_1}= \theta_M^2\left|\mathcal{K}_0\dot{X}_0(t) \right|^2_{R_1}\\
			&\hspace{8mm}-\theta_M \int_{t-r-\theta_M}^{t-r}e^{-2\delta_0(t-s)}\left|\mathcal{K}_0\dot{X}_0(s) \right|^2_{R_1}ds,\\
			&\dot{V}_{R_2}+2\delta_0V_{R_2}= \tau_M^2\left|\dot{X}_0(t) \right|^2_{R_2}\\
			&\hspace{8mm}-\tau_M \int_{t-\tau_M}^{t}e^{-2\delta_0(t-s)}\left|\dot{X}_0(s) \right|^2_{R_2}ds.
		\end{array}
	\end{equation*}
	By Jensen's and Park's inequalities (see \cite{Fridman14_TDS}) to obtain
	\begin{equation*}\label{eq:JensenParkDelay}
		\begin{array}{lll}
			&-\theta_M \int_{t-r-\theta_M}^{t-r}e^{-2\delta_0(t-s)}\left|\mathcal{K}_0\dot{X}_0(s) \right|^2_{R_1}ds\\
			&\leq -\varepsilon_{r+\theta_M} \scriptsize\begin{bmatrix}
				\mathcal{K}_0\Upsilon_u(t)\\ \mathcal{K}_0Q_u(t)
			\end{bmatrix}^T\begin{bmatrix}R_1 & G_1\\ * & R_1 \end{bmatrix}\begin{bmatrix}
				\mathcal{K}_0\Upsilon_u(t)\\ \mathcal{K}_0Q_u(t)
			\end{bmatrix},\\
			& -\tau_M \int_{t-\tau_M}^{t}e^{-2\delta_0(t-s)}\left|\dot{X}_0(s) \right|^2_{R_2}ds\\
			&\leq -\varepsilon_{\tau_M}\scriptsize\begin{bmatrix}
				\Upsilon_y(t)\\ Q_y(t)
			\end{bmatrix}^T\begin{bmatrix}R_2 & G_2\\ * & R_2 \end{bmatrix}\begin{bmatrix}
				\Upsilon_y(t)\\ Q_y(t)
			\end{bmatrix}.
		\end{array}
	\end{equation*}
	To compensate $\zeta(t-\tau_y)$ we use
	\begin{equation}\label{eq:HalanayDelay}
		\begin{array}{lll}
			&-2\delta_1\operatorname{sup}_{-\tau_M\leq \theta \leq 0}W(t+\theta)\leq -2\delta_1 V(t-\tau_y(t))\\
			&\overset{\eqref{eq:ZetaEstBoundartAct}}{\leq}-2\delta_1 \left[\Upsilon_y(t)+X_0(t)\right]^TP_0\left[\Upsilon_y(t)+X_0(t)\right]\\
			&-2\delta_1\left\|c \right\|_N^{-2} \zeta^2(t-\tau_y)-2\delta_1p_e\left|e^{N-N_0}(t) \right|^2_{e^{-2A_1\tau_y}}
		\end{array}
	\end{equation}
	where $\delta_0=\delta_1+\delta$.  Let $\eta(t) =\text{col} \left\{X_0(t),\zeta(t-\tau_y),\Upsilon_y(t),\right.$  $\left.Q_y(t),\mathcal{K}_0\Upsilon_u(t),\mathcal{K}_0\Upsilon_r(t),\mathcal{K}_0Q_u(t),e^{N-N_0}(t) \right\}$. Using \eqref{eq:VnomDelay} - \eqref{eq:HalanayDelay}, we employ Halanay's inequality
	\begin{equation*}\label{eq:FinalIneq}
		\begin{array}{ll}
			&\dot{W}(t)+2\delta_0 W(t) -2\delta_1\operatorname{sup}_{-\tau_M\leq \theta \leq 0}W(t+\theta)\\
			&\leq \eta^T(t)\Psi_{1} \eta(t) + 2\sum_{n=N+1}^{\infty}  \varpi_n z_n^2(t) \leq 0, \ t\geq 0,
		\end{array}
	\end{equation*}
	if
	\begin{equation*}\label{eq:PointActLMIsDelayed}
		\begin{array}{lll}
			\varpi_n=-\lambda_n+q+\delta_0+\left[\frac{1}{2\alpha}+\frac{1}{2\alpha_1}+\frac{1}{2\alpha_2} \right]\lambda_n < 0, \  n>N, \\
			\Psi_1 = \Psi_{\text{full}}+\Lambda^T \left[\mathcal{K}_0^T\left(r^2R_0+\theta_M^2R_1\right)\mathcal{K}_0+\tau_M^2R_2 \right]\Lambda<0.
		\end{array}
	\end{equation*}
	Here
	$$\Lambda = [\Lambda_{0}, \mathcal{L}_0C_1e^{-A_1\tau_y}], \
	\Gamma_3 =2p_e(A_1+\delta_0I-\delta_1e^{-2A_1\tau_y}),$$
	\begin{equation}\label{eq:PsiMatrix1}
		\begin{array}{lll}
			&\Psi_{\text{full}} = \left[
			\begin{array}{c|c}
				\Psi_{0} &  \Sigma_3 \\
				\hline
				* &  \Gamma_3
			\end{array}
			\right], \quad \Sigma_3 = \scriptsize \begin{bmatrix}
				P_0\mathcal{L}_0C_1e^{-A_1\tau_y}\\ 0
			\end{bmatrix}.
		\end{array}
	\end{equation}
	Monotonicity of $\left\{\lambda_n \right\}_{n=1}^{\infty}$
	and Schur's complement imply that $\varpi_n<0, \ n>N$ iff the second LMI in \eqref{eq:3.1LMI} holds.
	We have
	{\color{blue}$\Gamma_3 =2p_e(A_1+\delta_0I-\delta_1e^{-2A_1\tau_y})<0$} due to \eqref{eq:N0}.
	Therefore, by Schur complement for $p_e\to \infty$ we obtain that $\Psi_1<0$ iff \eqref{eq:PsiLMI} holds. \textcolor{blue}{Hence, feasibility of \eqref{eq:3.1LMI}, \eqref{eq:PsiLMI} and Lemma \ref{Lem:Halanay} lead to $W(t) \leq \exp \big(-2\delta_{\tau_M}t \big) \sup_{-\tau_M\leq \theta \leq 0} W(\theta)$ for $t \geq 0$. 
		The latter implies \eqref{eq:L2StabilityDelay}.}
	%
	\textcolor{blue}{Finally, note that \eqref{eq:3.1LMI} and \eqref{eq:PsiLMI} are \emph{reduced-order} LMIs whose dimension is independent of $N$. By arguments similar to Theorem 3.1 in  \cite{katz2020constructiveDelay}
		it can be shown that \eqref{eq:3.1LMI} and \eqref{eq:PsiLMI} are feasible for large enough $N$ and small enough $\tau_M,\theta_M,r$. Moreover, by Schur complements, the LMIs feasibility  for $N$ implies their feasibility for $N+1$.}
\end{proof}

\subsection{Predictor-based $L^2$-stabilization: known input delay}\label{sec_predictor}
In this section we compensate the constant and known part  $r$ of $\tau_u$ subject to \eqref{eq:tauUdecomp} by using a classical predictor \cite{Artstein82,selivanov2016observer}.
%
Recall the observer \eqref{eq:ZhatSeries0} which satisfies \eqref{eq:obsODEDelayednew}. Using the notations \eqref{eq:C0A0},\eqref{hatzN}, \eqref{eq:ErrDefNonDelayed0} and \eqref{eq:ErrDefNonDelayed00} we obtain
\begin{equation}\begin{array}{ll} \label{eq:z^N0Vector}
		&\dot {\hat{z}}^{N_0}(t)= A_0 \hat{z}^{N_0}(t)+B_0u (t-\tau_u)+L_0C_0e^{N_0}(t-\tau_y)\\
		&\hspace{7mm}+L_0C_1e^{N-N_0}(t-\tau_y)+{L}_0\zeta(t-\tau_y), \quad t\geq 0.
\end{array}\end{equation}
We propose the following predictor-based control law
\begin{equation}\label{eq_Pred}
	\begin{array}{lll}
		&\hspace{-2mm}\bar{z}(t)=e^{A_0r}\hat {z}^{N_0}(t)+\int_{t-r}^te^{A_0(t-s)}B_0u(s)ds,\ u(t)=K_0\bar {z}(t)
	\end{array}
\end{equation}
Differentiating $\bar {z}(t)$ and using \eqref{eq:z^N0Vector} we obtain
\begin{equation*}\label{ep18e}
	\begin{array}{ll}
		&\dot{\bar{z}}(t)=A_0\bar {z}(t)+B_0u(t)\\
		&+e^{A_0r}B_0\left[u(t-\tau_u)-u(t-r)\right]
		+e^{A_0r}L_0\\
		&\times\Big[C_0e^{N_0}(t-\tau_y)
		+C_1e^{N-N_0}(t-\tau_y)+\zeta(t-\tau_y)\Big].
	\end{array}
\end{equation*}

We present the reduced-order 
closed-loop system as
\begin{equation}\label{eq:bXDelay}
	\begin{array}{llllll}
		\dot{\bar X}(t) = & \bar{F}_0 \bar X(t)+\bar{\mathcal{B}}_0\mathcal{K}_0\bar{\Upsilon}_u(t)+\bar{\mathcal{L}}_0 \mathcal{C}_0\bar{\Upsilon}_y(t)\\&
		+\bar{\mathcal{L}}_0\zeta(t-\tau_y)+\bar{\mathcal{L}}_0C_1e^{-A_1\tau_y}e^{N-N_0}(t),\\
		\dot{z}_n(t)=& (-\lambda_n+q)z_n(t)+b_n\mathcal{K}_0\bar X(t)\\
		&+b_n \mathcal{K}_0[\bar{\Upsilon}_u(t) +\bar{\Upsilon}_r(t)], \quad n>N, \quad t\geq0,
	\end{array}
\end{equation}
where
\begin{equation*}
	\begin{array}{lll}
		&\bar X(t)=\text{col}\{\bar z(t), e^{N_0}(t)\},\ \bar{\Upsilon}_{y}(t) = \bar X(t-\tau_y)-\bar X(t),\\
		&\bar{\Upsilon}_u(t) = \bar X(t-\tau_u)-\bar X(t-r), \ \mathcal{C}_0 = [0_{1\times(N_0+1)},C_0],\\
		&\bar{\Upsilon}_r(t) = \bar{X}(t-r)-\bar{X}(t),\ \bar{\mathcal{L}}_0 = \text{col}\left\{e^{A_0r}L_0, -L_0 \right\},\\
		&\bar{Q}_u(t) = \bar{X}(t-r-\theta_M)-\bar{X}(t-\tau_u),\\
		&\bar{Q}_y(t) = \bar{X}(t-\tau_M)-\bar{X}(t-\tau_y),
	\end{array}
\end{equation*}
\begin{equation}\label{eq:PredictorVectors}
	\begin{array}{lll}
		& \bar{F}_0 = \scriptsize\begin{bmatrix}A_0+B_0K_0 & e^{A_0r}L_0C_0 \\ 0 & A_0-L_0C_0 \end{bmatrix},\ \bar{\mathcal{B}}_0= \text{col}\left\{e^{A_0r}B_0,0\right\}.
	\end{array}
\end{equation}
As in the non-delayed case, here $e^{N-N_0}(t)$ satisfies \eqref{eNN0} and is exponentially decaying, whereas  $\zeta(t)$ satisfies \eqref{eq:ZetaEstBoundartAct}. From \eqref{eq_Pred} we have that exponential decay of $\bar X(t)$ implies exponential decay of $X_0(t)$ in \eqref{eq:ErrDefNonDelayed00}.

For $L^2$-stability analysis of \eqref{eq:bXDelay}, \eqref{eNN0} we fix $\delta_0>\delta$ and define the Lyapunov functional \eqref{eq:VComponents0}. Here $V(t)$ and  $V_{S_i}, V_{R_i}, \ i\in \left\{0,1,2\right\}$ are given by \eqref{eq:PointActVNonDelayed_L2} and \eqref{eq:VComponentsR}, respectively, with $X_0$ replaced by $\bar X$.

\textcolor{blue}{To state the main result of this section, let $G_1\in \mathbb{R}$ and $G_2\in \mathbb{R}^{2(N_0+1)\times2(N_0+1)}$ and $0<\alpha,\alpha_1,\alpha_2\in \mathbb{R}$. We introduce}
\begin{equation*}
	\begin{array}{lll}
		&\hspace{-3mm}\bar{\Psi}_{0} = \scriptsize \left[
		\begin{array}{c|c}
			\bar{\Theta} & \bar{\Sigma}_1 \ \ \bar{\Sigma}_2  \\
			\hline
			* & \operatorname{diag}\left\{\Gamma_1, \Gamma_2 \right\}
		\end{array}
		\right], \ \bar{\Theta} = \scriptsize\begin{bmatrix}\bar{\Phi} & P_0\bar{\mathcal{L}}_0\\
			* & -2\delta_1\left\|c \right\|_N^{-2} \end{bmatrix},\\
		&\hspace{-3mm}\bar{\Sigma}_1 =\scriptsize\begin{bmatrix}
			P_0\bar{\mathcal{L}}_0\mathcal{C}_0-2\delta_1P_0-\varepsilon_MS_2& -\varepsilon_MS_2\\0 & 0
		\end{bmatrix},\\
		&\hspace{-3mm}\bar{\Sigma}_2 = \scriptsize \begin{bmatrix}
			P_0\bar{\mathcal{B}}_0-\varepsilon_{r,M}\mathcal{K}_0^TS_1 \ & \bar{\Xi}_1  \ & -\varepsilon_{r,M}\mathcal{K}_0^TS_1\\ 0 \ & 0 \ & 0
		\end{bmatrix},\\
		&\hspace{-3mm}\bar{\Phi} = P_0\bar{F}_0+\bar{F}_0^TP_0+2\delta P_0+(1-\varepsilon_r)\mathcal{K}_0^TS_0\mathcal{K}_0\\
		&+\frac{2\alpha}{\pi^2N}\mathcal{K}_0^T\mathcal{K}_0+(\varepsilon_r-\varepsilon_{r,M})\mathcal{K}_0^TS_1\mathcal{K}_0+ (1-\varepsilon_M)S_2,\\
		&\hspace{-3mm}\bar{\Xi}_1 =-\varepsilon_{r}\mathcal{K}_0^TS_0+(\varepsilon_r - \varepsilon_{r,M})\mathcal{K}_0^TS_1,\\
		&\hspace{-3mm} \bar{\Lambda}_{0} = [\bar{F}_0, \bar{\mathcal{L}}_0, \bar{\mathcal{L}}_0\mathcal{C}_0, 0 ,\bar{\mathcal{B}}_0,0,0]
	\end{array}
\end{equation*}
where $\Gamma_i , \ i\in \left\{1,2,3\right\}$ and $\varepsilon_{\tau}, \ \tau \in \left\{r,\tau_M,r+\theta_M\right\}$ are given in \eqref{eq:PsiMatrix2}, \eqref{eq:PsiMatrix1}.
\begin{theorem}\label{Thm:NonLocalL2DelayPredB}
	\textcolor{blue}{Consider \eqref{eq:PDEPointActIntervalDelayed}, measurement \eqref{eq:BoundMeas} with $c\in L^2(0,1)$ satisfying \eqref{eq:AsscnNonDelayed}, control law \eqref{eq_Pred}. Let $\delta_0>\delta>0$ and $\delta_1 = \delta_0-\delta$. Let $N_0\in \mathbb{Z}_+$ satisfy \eqref{eq:N0} and $N\geq N_0+1$. Assume that $L_0$ and $K_0$ are subject to \eqref{eq:GainsDesignL}  and \eqref{eq:GainsDesignK}, respectively. Given $r,\theta_M,\tau_M>0$, let there exist positive definite matrices $P_0,S_2,R_2\in \mathbb{R}^{2(N_0+1)\times 2(N_0+1)}$, scalars $S_0,R_0,S_1,R_1, \alpha,\alpha_1,\alpha_2>0$, $G_1\in \mathbb{R}$ and $G_2\in \mathbb{R}^{2(N_0+1)\times 2(N_0+1)}$ such that \eqref{eq:3.1LMI} and
		\begin{equation}\label{eq:PsireducedPredB}
			\begin{array}{lll}
				&\bar{\Psi}_{0} +\bar{\Lambda}_{0}^T \left[\mathcal{K}_0^T\left(r^2R_0+\theta_M^2R_1\right)\mathcal{K}_0+\tau_M^2R_2 \right]\bar{\Lambda}_{0}\normalsize<0.
			\end{array}
		\end{equation}	
		hold. Then the solution $z(x,t)$ to \eqref{eq:PDEPointActIntervalDelayed} under the control law \eqref{eq_Pred} and the corresponding observer $\hat z(x,t)$ defined by \eqref{eq:ZhatSeries0}, \eqref{eq:obsODEDelayednew}
		satisfy \eqref{eq:L2StabilityDelay} for some $M>0$ and $\delta_{\tau_M}>0$ defined by \eqref{eq:DecRate}.  Moreover, LMIs \eqref{eq:3.1LMI} and \eqref{eq:PsireducedPredB} are always feasible if $N$ is large enough and $\tau_M,\theta_M,r$ are small enough. Feasibility of \eqref{eq:3.1LMI} and \eqref{eq:PsireducedPredB} for $N$ implies their feasibility for $N+1$.}
\end{theorem}
\begin{proof}
	\textcolor{blue}{The proof is essentially identical to proof of Theorem \ref{Thm:NonLocalL2Delay}. Hence, we only state the differences. Let $\eta(t) = \text{col}\left\{\bar{X}(t),\zeta(t-\tau_y),\bar{\Upsilon}_y(t),\bar{Q}_y(t),\mathcal{K}_0\bar{\Upsilon}_u(t),\right.$ $\left.\mathcal{K}_0\bar{\Upsilon}_r(t),\mathcal{K}_0\bar{Q}_u(t),e^{N-N_0}(t) \right\}$.  By arguments similar to \eqref{eq:VnomDelay}-\eqref{eq:HalanayDelay} we obtain
		\begin{equation}\label{eq:FinalIneqPredB}
			\begin{array}{ll}
				&\dot{W}(t)+2\delta_0 W(t) -2\delta_1\operatorname{sup}_{-\tau_M\leq \theta \leq 0}W(t+\theta)\\
				&\leq \eta^T(t)\Psi_{2} \eta(t) + 2\sum_{n=N+1}^{\infty}  \varpi_n z_n^2(t) \leq 0, \ t\geq 0,
			\end{array}
		\end{equation}
		if
		\begin{equation}\label{eq:PointActLMIsPredB}
			\begin{array}{lll}
				\varpi_n=-\lambda_n+q+\delta_0+\left[\frac{1}{2\alpha}+\frac{1}{2\alpha_1}+\frac{1}{2\alpha_2} \right]\lambda_n < 0, \  n>N, \\
				\Psi_2 = \bar{\Psi}+\bar{\Lambda}^T \left[\mathcal{K}_0^T\left(r^2R_0+\theta_M^2R_1\right)\mathcal{K}_0+\tau_M^2R_2 \right]\bar{\Lambda}<0.
			\end{array}
		\end{equation}
		Here $\bar{\Lambda} = [\bar{\Lambda}_{0}, \bar{\mathcal{L}}_0C_1e^{-A_1\tau_y}]$ and
		\begin{equation}\label{eq:PsiMatrixPredB}
			\begin{array}{lll}
				&\bar{\Psi} = \scriptsize\left[
				\begin{array}{c|c}
					\bar{\Psi}_{0} &  \Sigma_3 \\
					\hline
					* &  \Gamma_3
				\end{array}
				\right], \ \Sigma_3 = \scriptsize \begin{bmatrix}
					P_0\bar{\mathcal{L}}_0C_1e^{-A_1\tau_y}\\ 0
				\end{bmatrix}.
			\end{array}
		\end{equation}
		Monotonicity of $\left\{\lambda_n \right\}_{n=1}^{\infty}$
		and Schur's complement imply that $\varpi_n<0, \ n>N$ iff the second LMI in \eqref{eq:3.1LMI} holds. Finally, note that \eqref{eq:N0} implies $\Gamma_3<0$. 
		By Schur complement and $p_e\to \infty$,  $\Psi_2<0$ iff \eqref{eq:PsireducedPredB} holds. Note that \eqref{eq:3.1LMI} and \eqref{eq:PsireducedPredB} are again of reduced-order (i.e, the dimension is independent of $N$).}
\end{proof}

\subsection{Predictor-based $L^2$-stabilization: unknown input delay}\label{Sec:PredictorKnownDelay}
In this section we assume an input delay $\tau_u(t)=r+\theta(t)$ with a known constant part $r>0$ and unknown $\theta(t)\in[0,\theta_M]$. Since $\theta(t)$ is unknown, the observer \eqref{eq:ZhatSeries0} is designed to satisfy \eqref{eq:obsODEDelayednew} with $u(t-\tau_u)$ \emph{replaced} by $u(t-r)$. Therefore, \eqref{eq:z^N0Vector} is modified as follows:
\begin{equation}\begin{array}{ll} \label{eq:zN0PredC}
		\dot {\hat{z}}^{N_0}(t)=& A_0 \hat{z}^{N_0}(t)+B_0u (t-r)+L_0C_0e^{N_0}(t-\tau_y)\\
		&+L_0C_1e^{N-N_0}(t-\tau_y)+{L}_0\zeta(t-\tau_y)
\end{array}\end{equation}
whereas $\hat z^{N-N_0}(t)$ satisfies
\begin{equation}\label{eq:zNminN0PredC}
	\dot {\hat{z}}^{N-N_0}(t)= A_1 \hat{z}^{N-N_0}(t)+B_1u(t-r).
\end{equation}
Furthermore, the  estimation error satisfies
\begin{equation}\label{eq:errorPredC}
	\begin{array}{ll}
		&\dot e^{N_0}(t)=A_0e^{N_0}(t)+B_0[u(t-\tau_u)-u(t-r)]\\
		&\quad -L_0[C_0e^{N_0}(t-\tau_y)+C_1 e^{N-N_0}(t-\tau_y)+\zeta(t-\tau_y)],
		\\
		&\dot e^{N-N_0}(t)=A_1 e^{N-N_0}(t)+B_1[u(t-\tau_u)-u(t-r)].
	\end{array}
\end{equation}
As in \cite{katz2020constructiveDelay}, uncertainty in $\tau_u$ leads to \emph{coupling} of $e^{N-N_0}(t)$  with $u(t)$. We propose the predictor-based control law \eqref{eq_Pred}. Differentiating $\bar {z}(t)$ and using \eqref{eq:zN0PredC} we obtain
\begin{equation}\label{eq:PredPredC}
	\begin{array}{ll}
		&\dot{\bar{z}}(t)=(A_0+B_0K_0)\bar {z}(t)
		+e^{A_0r}L_0\\
		&\times\Big[C_0e^{N_0}(t-\tau_y)
		+C_1e^{N-N_0}(t-\tau_y)+\zeta(t-\tau_y)\Big]
	\end{array}
\end{equation}

Differently from the case of a known $\tau_u$, we introduce
\begin{equation}\label{eq:NewXPredC}
	\bar{X}(t)=\text{col}\{\bar z(t), e^{N_0}(t),e^{N-N_0}(t)\}
\end{equation}
as the closed-loop state, which includes $e^{N-N_0}(t)$. Note that differently from \cite{katz2020constructiveDelay}, $\hat z^{N-N_0}(t)$ is not a part of $\bar{X}(t)$. Therefore, for a given $N$, the LMIs subsequently obtained will not be of reduced-order, but are of essentially smaller dimension than in \cite{katz2020constructiveDelay}. Recall $\bar{Q}_u(t)$, $\bar{Q}_y(t)$, $\bar{\Upsilon}_u(t)$, $\bar{\Upsilon}_y(t)$ and $\bar{\Upsilon}_r(t)$ given in \eqref{eq:PredictorVectors} and let
\begin{equation*}\label{eq:NotationsPredC}
	\begin{array}{lll}
		& \bar F = \scriptsize\begin{bmatrix}A_0+B_0K_0 & e^{A_0r}L_0C_0 & e^{A_0r}L_0C_1\\ 0 & A_0-L_0C_0 & -L_0 C_1\\
			0 & 0 & A_1 \end{bmatrix},\\
		& \bar{\mathcal{L}} = \text{col}\left\{e^{A_0r}L_0, -L_0, 0\right\}, \ \mathcal{C} = [0_{1\times(N_0+1)},\ C_0, \ C_1],\\
		& \bar {\mathcal{B}}= \text{col}\left\{0_{(N_0+1) \times 1}, B_0, B_1\right\}, \ \mathcal{K}_0 = [K_0, \ 0, \ 0].
	\end{array}
\end{equation*}
The closed-loop system is governed by
\begin{equation}\label{eq:ClosedLoopPredC}
	\begin{array}{llllll}
		\dot{\bar X}(t) = & \bar{F} \bar{X}(t)+\bar{\mathcal{B}}\mathcal{K}_0\bar{\Upsilon}_u(t)+\bar{\mathcal{L}} \mathcal{C}\bar{\Upsilon}_y(t) +\bar{\mathcal{L}}\zeta(t-\tau_y),\\
		\dot{z}_n(t)=& (-\lambda_n+q)z_n(t)+b_n\mathcal{K}_0\bar{X}(t)\\
		&+b_n \mathcal{K}_0[\bar{\Upsilon}_u(t)+\bar{\Upsilon}_r(t)], \quad n>N, \quad t\geq 0
	\end{array}
\end{equation}
where $\zeta(t)$ satisfies \eqref{eq:ZetaEstBoundartAct}. From \eqref{eq:zNminN0PredC} follows that $\hat{z}^{N-N_0}(t)$ is exponentially decaying if the closed-loop system \eqref{eq:ClosedLoopPredC} is exponentially decaying.

For $L^2$-stability of the closed-loop system \eqref{eq:ClosedLoopPredC} let $\delta_0>\delta$ and define the Lyapunov functional \eqref{eq:VComponents0} with $V(t)$ replaced by $V_0(t)$, given in \eqref{eq:PointActVNonDelayed_L2}, $V_{S_i}, V_{R_i}, \ i\in \left\{0,1,2\right\}$ given in \eqref{eq:VComponentsR} and $X_0(t)$ is replaced by $\bar{X}(t)$ everywhere.
\textcolor{blue}{To state the main result of this section, let $G_1\in \mathbb{R}$ and $G_2\in \mathbb{R}^{(N+N_0+2)\times(N+N_0+2)}$ and $0<\alpha,\alpha_1,\alpha_2\in \mathbb{R}$. Let}
\begin{equation*}
	\begin{array}{lll}
		&\bar{\Psi}_1 = \scriptsize\left[
		\begin{array}{c|c }
			\bar{\Psi}_{2} &  \Sigma_4 \ \ \Sigma_5 \\
			\hline
			* &  \operatorname{diag}\left\{\Gamma_1,\Gamma_2\right\}
		\end{array}
		\right], \ \bar{\Psi}_2 = \scriptsize\begin{bmatrix}\bar{\Phi}_1 & P_0\bar{\mathcal{L}}\\
			* & -2\delta_1\left\|c \right\|_N^{-2} \end{bmatrix}  \\
		& \Sigma_4 =\scriptsize\begin{bmatrix}
			P_0\bar{\mathcal{L}}\mathcal{C}-2\delta_1P_0-\varepsilon_MS_2& \quad -\varepsilon_MS_2\\0 & 0
		\end{bmatrix},\\
		& \Sigma_5 = \scriptsize \begin{bmatrix}
			P_0\bar{\mathcal{B}}-\varepsilon_{r,M}\mathcal{K}_0^TS_1 \ & \bar{\Xi}_1  \ & -\varepsilon_{r,M}\mathcal{K}_0^TS_1\\ 0 \ & 0 \ & 0
		\end{bmatrix}, \\
	\end{array}
\end{equation*}
\begin{equation*}
	\begin{array}{lll}
		&\bar{\Phi}_1 = P_0\bar{F}+\bar{F}^TP_0+2\delta P_0+(1-\varepsilon_r)\mathcal{K}_0^TS_0\mathcal{K}_0\\
		&+\frac{2\alpha}{\pi^2N}\mathcal{K}_0^T\mathcal{K}_0+(\varepsilon_r-\varepsilon_{r,M})\mathcal{K}_0^TS_1\mathcal{K}_0+ (1-\varepsilon_M)S_2,\\
		&\bar{\Lambda}_1 = [\bar{F},\bar{\mathcal{L}}, \bar{\mathcal{L}}\mathcal{C},0,\bar{\mathcal{B}},0,0]
	\end{array}
\end{equation*}
with $\Gamma_1$, $\Gamma_2$ given in \eqref{eq:PsiMatrix2}.
\begin{theorem}\label{Thm:NonLocalL2DelayPredC}
	Consider \eqref{eq:PDEPointActIntervalDelayed} with \emph{unknown} input delay $\tau_u(t)$, measurement \eqref{eq:BoundMeas} with $c\in L^2(0,1)$ satisfying \eqref{eq:AsscnNonDelayed}, control law \eqref{eq_Pred}. Let $\delta_0>\delta>0$ and $\delta_1 = \delta_0-\delta$. Let $N_0\in \mathbb{Z}_+$ satisfy \eqref{eq:N0} and $N\geq N_0+1$. Let $L_0$ and $K_0$ satisfy \eqref{eq:GainsDesignL}  and \eqref{eq:GainsDesignK}, respectively. Given $r,\theta_M,\tau_M>0$, let there exist positive definite matrices $P_0,S_2,R_2\in \mathbb{R}^{(N+N_0+2)\times (N+N_0+2)}$, scalars $S_0,R_0,S_1,R_1, \alpha,\alpha_1,\alpha_2>0$, $G_1\in \mathbb{R}$ and $G_2\in \mathbb{R}^{(N+N_0+2)\times (N+N_0+2)}$ such that \eqref{eq:3.1LMI} and
	\begin{equation}\label{eq:Psi3Unknown}
		\bar{\Psi}_1+\bar{\Lambda}_1^T \left[\mathcal{K}_0^T\left(r^2R_0+\theta_M^2R_1\right)\mathcal{K}_0+\tau_M^2R_2 \right]\bar{\Lambda}_1<0
	\end{equation}
	hold. Then the solution $z(x,t)$ to \eqref{eq:PDEPointActIntervalDelayed} under the control law \eqref{eq_Pred} and the observer $\hat z(x,t)$ defined by \eqref{eq:ZhatSeries0}, \eqref{eq:zN0PredC} and \eqref{eq:zNminN0PredC}
	satisfy \eqref{eq:L2StabilityDelay} for some $M>0$ and $\delta_{\tau_M}>0$ defined by \eqref{eq:DecRate}. The LMIs \eqref{eq:3.1LMI} and \eqref{eq:Psi3Unknown} are always feasible if $N$ is large enough and $\tau_M,\theta_M,r$ are small enough.
\end{theorem}
\begin{proof}
	The proof is essentially identical to proof of Theorem \ref{Thm:NonLocalL2Delay}. Hence, we only state the differences. Let $\eta(t) = \text{col}\left\{\bar{X}(t),\zeta(t-\tau_y),\bar{\Upsilon}_y(t),\bar{\mu}_y(t),\mathcal{K}_0\bar{\Upsilon}_u(t),\mathcal{K}_0\bar{\Upsilon}_r(t)\right. $ $\left.,\mathcal{K}_0\bar{Q}_u(t)\right\}$. Similar to \eqref{eq:VnomDelay}-\eqref{eq:HalanayDelay} we obtain
	\begin{equation*}\label{eq:FinalIneqPredC}
		\begin{array}{ll}
			&\dot{W}(t)+2\delta_0 W(t) -2\delta_1\operatorname{sup}_{-\tau_M\leq \theta \leq 0}W(t+\theta)\\
			&\leq \eta^T(t)\Psi_{3} \eta(t) + 2\sum_{n=N+1}^{\infty}  \varpi_n z_n^2(t) \leq 0, \ t\geq 0,
		\end{array}
	\end{equation*}
	if
	\begin{equation}\label{eq:PointActLMIsPredC}
		\begin{array}{lll}
			\hspace{-1mm}\varpi_n=-\lambda_n+q+\delta_0+\left[\frac{1}{2\alpha}+\frac{1}{2\alpha_1}+\frac{1}{2\alpha_2} \right]\lambda_n < 0, \  n>N, \\
			\hspace{-1mm}\Psi_3 = \bar{\Psi}_1+\bar{\Lambda}_1^T \left[\mathcal{K}_0^T\left(r^2R_0+\theta_M^2R_1\right)\mathcal{K}_0+\tau_M^2R_2 \right]\bar{\Lambda}_1<0.
		\end{array}
	\end{equation}
	Monotonicity of $\left\{\lambda_n \right\}_{n=1}^{\infty}$
	and Schur's complement imply that $\varpi_n<0, \ n>N$ iff the second LMI in \eqref{eq:3.1LMI} holds, whereas $\Psi_3<0$ is exactly \eqref{eq:Psi3Unknown}.
\end{proof}
\section{Example: temperature control in a rod}
Consider control of heat flow in the  rod  with constant thermal conductivity, mass density, specific heat and  reaction coefficient
\cite{christofides2001,curtain2009transfer}.  The control action effects the heat flow at one end, while keeping the heat flow in the other end fixed.
The model of spatiotemporal evolution of the dimensionless rod temperature
(denoted by $z(x,t)$) is given by \eqref{eq:PDE1PointActNonDelayed}, where $q$ is the reaction coefficient. We consider $q=3$, which results in an unstable open-loop system. The measurement of
the distributed rod temperature  is given by \eqref{eq:BoundMeasNonDelayed}, where
$c(x) = \chi_{[0.3,0.9]}(x)$ (i.e, the indicator function of $[0.3,0.9]$).
The control objective is to stabilize the rod temperature at the unstable steady state $z(x,t)=0$.

The observer and controller gains are found from \eqref{eq:GainsDesignL} and \eqref{eq:GainsDesignK}. For non-delayed stabilization we consider $\delta\in \left\{0.1,1,2,5\right\}$ which result in $N_0=0$. For each $\delta$ we compute the corresponding 
gains and find the minimum value of $N$ such that the LMI of Theorem \ref{Thm:PointActNonDelayed} holds (see Table \ref{Tab:NoDelay1}).
\begin{table}[h]
	\begin{center}
		\footnotesize\begin{tabular}{|c|c|c|c|c|c|}
			\hline
			$\delta$ & $0.1$ & $1$& $2$& $5$ &\textcolor{blue}{7.5}\\
			\hline
			$N$ & $3$ & $4$ & $4$ & $4$ & \textcolor{blue}{$5$}\\
			\hline
			$K_0$ & $-5$ & $-5$ & $-7$ & $-13$ & \textcolor{blue}{$-18$}\\	
			\hline
			$L_0$ & $5.5$ & $8.33$ & $11.67$ & $21.6$& \textcolor{blue}{$29.8$}\\
			\hline		
		\end{tabular}
	\end{center}
	\caption{\label{Tab:NoDelay1} {\color{blue} Minimal $N$ that guarantees decay rate $\delta$: non-delayed case.}}
\end{table}

For delayed stabilization we choose $\delta = 0$, which results in $N_0=0$. The controller and observer
gains are given by
\begin{equation}\label{eq:ExGains}
	K_0 = -5.5, \quad L_0 = 5.5.
\end{equation}
We  verify the feasibility of LMIs of {\color{blue}Theorems \ref{Thm:NonLocalL2Delay} (no predictor), 
	\ref{Thm:NonLocalL2DelayPredB} (predictor, known $\tau_u$) 
	and \ref{Thm:NonLocalL2DelayPredC} (predictor, unknown $\tau_u$)} 
for $\delta_0=\delta_1$. Since the corresponding LMIs are strict, feasibility with $\delta=0$ implies their feasibility for small enough $\delta_*>0$. In the first test we fix $\tau_M = \theta_M = 10^{-7}$ and find the minimal value of $N$ which guarantees the feasibility of the LMIs for increasing values of $r$. The results are given in Table \ref{Tab:NoTauTheta1}. It is seen  that   predictor allows to increase the maximal value of $r$
from $0.14$ till $0.3$. \textcolor{blue}{The maximum value of $r$, with corresponding $N$ for which the LMIs of Theorems \ref{Thm:NonLocalL2Delay}, \ref{Thm:NonLocalL2DelayPredB} and \ref{Thm:NonLocalL2DelayPredC} were found feasible are $r = 0.16 \ (N=18)$, $r=0.44 \ (N = 24)$ and $r=0.41\ (N = 26)$, respectively.}
\begin{table}[h]
	\begin{center}
		\footnotesize\begin{tabular}{|c|c|c|c|c|c|c|}
			\hline
			$r$ & 0.06 & 0.1 &  0.14 & 0.18 & 0.26 & 0.3\\
			\hline
			{\color{blue} Th. \ref{Thm:NonLocalL2Delay}: no predictor}
			& 6 & 6 & 14 & - & - & -\\
			\hline	
			{\color{blue} Th. \ref{Thm:NonLocalL2DelayPredB}, \ref{Thm:NonLocalL2DelayPredC}: predictor} & 6 & 6 & 6 & 8 & 12 & 16\\
			\hline
		\end{tabular}
	\end{center}
	\caption{\label{Tab:NoTauTheta1} {\color{blue} Minimal $N$ for the stability: given $r$ and $\tau_M=\theta_M = 10^{-7}$.}}
\end{table}

In the second test we fix $\tau_M=\theta_M$ and find the maximum value of $r$ and the corresponding minimal value of $N$ for which  LMIs are feasible. The results are given in Table \ref{Tab:PosTauTheta}. It is seed that for $\theta_M=\tau_M=0.01$ the LMIs of Theorems \ref{Thm:NonLocalL2DelayPredB} and  \ref{Thm:NonLocalL2DelayPredC} allow for larger $r$ than in Theorem \ref{Thm:NonLocalL2Delay}. For $\theta_M=\tau_M=0.04$ the same comparison holds only for
Theorems \ref{Thm:NonLocalL2DelayPredB} and Theorem \ref{Thm:NonLocalL2Delay},
whereas no feasibility was obtained in Theorem \ref{Thm:NonLocalL2DelayPredC} due to higher-dimensional LMIs for $N=30$.
\begin{table}[h]
	\begin{center}
		\footnotesize\begin{tabular}{|c|c|c|}
			\hline
			$\tau_M=\theta_M$ & $
			0.01$ & $
			0.04$ \\
			\hline
			{\color{blue} Th. \ref{Thm:NonLocalL2Delay}: no predictor} & $r=0.14$, $N=30$ & $r=0.12$, $N=30$ \\		
			\hline
			{\color{blue} Th. \ref{Thm:NonLocalL2DelayPredB}: predictor} & $r=0.3$, $N=30$ & $r=0.25$, $N=30$ \\	
			\hline
			{\color{blue} Th. \ref{Thm:NonLocalL2DelayPredC}: predictor}& $r=0.25$, $N=22$ & - \\
			\hline
		\end{tabular}
	\end{center}
	\caption{\label{Tab:PosTauTheta}  Maximal $r$ and minimal $N$ that guarantee the stability.}
\end{table}

Our reduced-order LMIs are feasible for larger values of $N$ than in \cite{katz2020constructiveDelay} (where for $N>9$ we could not verify LMIs) due to a significantly lower computational complexity. A larger $N$ allows larger delays in example. {\color{blue} For additional LMI simulations with different gains see \cite{katz2020Reduced}).}

For simulations of the solutions to the closed-loop systems we choose observer and controller gains given by \eqref{eq:ExGains}. We fix $\tau_M=\theta_M=0.01$ and choose the \emph{known} delays $\tau_u(t) = r+0.01\sin^2(120t)$ and $\tau_y(t) = 0.01\cos^2(120t)$. Note that $\dot{\tau}_y<1$ and $ \dot{\tau}_u<1$ \emph{does not} hold. We choose $r_{\text{max}}$ and $N$ given in the first column and the first two lines of Table \ref{Tab:PosTauTheta}. For the initial condition $z(x,0)= 10x^2(1-x)^2$ we do simulations of the closed-loop systems \eqref{eq:X0Delay} (without predictor) and \eqref{eq:bXDelay} (with predictor) and the ODEs satisfied by $\hat{z}^{N-N_0}(t)$. In both cases, we simulate the ODEs of $z_n(t)$ for $N+1\leq n\leq 50$. The value of $\zeta(t)$, given by \eqref{eq:IntroZetaNonDelayed}, is approximated by $\zeta(t) \approx \sum_{n=N+1}^{50}c_nz_n(t)$. Results of the simulations are given at the top of Figure \ref{fig:Fig1} and confirm our theoretical results. Moreover, a simulation for $r = 0.22$ and $N=30$ without predictor shows
instability (see the bottom of Figure \ref{fig:Fig1}). The use of predictor allows to stabilize for a larger $r=0.3$ with  $N=30$.
\begin{figure}
	\centering
	\includegraphics[width=80mm,scale=1]{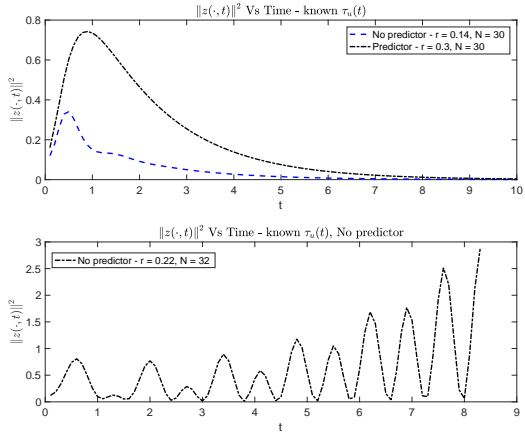}
	\vspace{-0.3cm}
	\caption{{\color{blue} Simulation results for known $\tau_u$. Top: stability confirming the LMI results. Bottom: instability without predictor}
}\label{fig:Fig1}
\end{figure}

\section{Conclusion}
\vspace{-0.3cm}
We suggested a finite-dimensional observer-based control of the 1D heat equation under Neumann actuation, non-local measurement and  fast-varying input/output delays.
Reduced-order LMI stability conditions were derived. 
Classical predictors were used to  enlarge the delays.

\bibliographystyle{abbrv}
\bibliography{Bibliography021218}

\end{document}